\RequirePackage[l2tabu, orthodox]{nag}
\documentclass[a4paper,11pt,reqno]{amsart}

\usepackage{amssymb,amscd,amsthm}
\usepackage{mathrsfs}
\usepackage{hyperref}

\usepackage[ansinew]{inputenc}
\usepackage[centertags]{amsmath}
\usepackage{epsfig,graphicx,enumerate}

\usepackage{xcolor}

\newcommand{\NN}{\mbox{$\mathbb{N}$}}

\newcommand{\R}{\mbox{$\mathbb{R}$}}

\newcommand{\GL}{\mathrm{GL}}

 \def\NN{{\mathbb N}}  \def\PP{{\mathbb P}}
 \def\RR{{\mathbb R}}  
  
 \def\ZZ{{\mathbb Z}}

    \def\cV{\mathcal{V}}
    \def\cW{\mathcal{W}}
    \def\cX{\mathcal{X}}
\def\cY{\mathcal{Y}}  \def\cZ{\mathcal{Z}}

\newcommand{\g}{\mathfrak{g}}

\newcommand{\aaa}{\mathfrak{a}}

\newcommand{\Isom}{\mathrm{Isom}}
\newcommand{\Shad}{\mathrm{Shad}}

\newcommand{\ie}{i.e.\ }
\newcommand{\eg}{e.g.\ }
\newcommand{\resp}{resp.\ }

\newtheorem{teo}{Theorem}[section]

\newtheorem{cor}[teo]{Corollary}
\newtheorem{lema}[teo]{Lemma}
\newtheorem{prop}[teo]{Proposition}
\newtheorem{fact}[teo]{Fact}
\newtheorem{setting}[teo]{Setting}

\theoremstyle{definition}
\newtheorem{defi}[teo]{Definition}
\theoremstyle{remark}
\newtheorem{remark}[teo]{Remark}

\newtheorem{ex}[teo]{Example}

\numberwithin{equation}{section}

\newcommand{\eps}{\varepsilon}

\newcommand*{\longhookrightarrow}{\ensuremath{\lhook\joinrel\relbar\joinrel\rightarrow}}

\author[F. Kassel]{Fanny Kassel}
\address{CNRS and Laboratoire Alexander Grothendieck, Institut des Hautes \'Etudes Scientifiques, Universit\'e Paris-Saclay, 35 route de Chartres, 91440 Bures-sur-Yvette, France}
\email{kassel@ihes.fr}

\author[R. Potrie]{Rafael Potrie}
\address{CMAT, Facultad de Ciencias, Universidad de la Rep\'ublica, Igu\'a 4225, Montevideo, 11400, Uruguay, and IRL2030 - IFUMI `Laboratorio del Plata' CNRS, Montevideo}
\urladdr{www.cmat.edu.uy/$\sim$rpotrie}
\email{rpotrie@cmat.edu.uy}

\title{A simultaneous Abels--Margulis--Soifer lemma}
\thanks{This project received funding from the European Research Council (ERC) under the European Union's Horizon 2020 research and innovation programme (ERC starting grant DiGGeS, grant agreement No.\ 715982). R.P.\ was partially supported by CSIC and ANII}

\begin{document}

\maketitle

\begin{abstract}
The Abels--Margulis--Soifer lemma states that if a semigroup $\Gamma$ acts strongly irreducibly by linear transformations on a finite-dimensional real vector space, then any element of~$\Gamma$ can be multiplied by an element of some fixed finite subset of~$\Gamma$ so that it becomes proximal (\ie it acts on the corresponding projective space with an attracting fixed point and a repelling projective hyperplane) and even uniformly proximal (\ie the distance between the attracting fixed point and the repelling projective hyperplane is uniformly bounded from below and the contraction towards the attracting fixed point is uniformly strong).
We prove a version of this lemma simultaneously for linear representations of a semigroup~$\Gamma$, acting on the corresponding projective spaces, and for representations of~$\Gamma$ to isometry groups of (not necessarily proper) Gromov hyperbolic metric spaces, acting on the corresponding Gromov boundaries.
\end{abstract}
  
\medskip

\section{Introduction} \label{sec:intro}

\subsection{The classical Abels--Margulis--Soifer lemma} \label{subsec:intro-classical-AMS}

Let $V$ be a finite-dimensional real vector space, of dimension $\geq 2$, and let $\PP(V)$ be the corresponding projective space.
Recall that an element $g\in\GL(V)$ is said to be \emph{proximal in $\PP(V)$} if it admits an attracting fixed point $x_g^+$ in $\PP(V)$, \ie a fixed point $x_g^+$ with a neighborhood $\cV_g$ in $\PP(V)$ such that $g^n\cdot x\to x_g^+$ as $n\to +\infty$ for all $x\in\cV_g$; equivalently, $g$ has a unique complex eigenvalue of maximal modulus, with multiplicity~$1$ (such an eigenvalue is then necessarily real, and $x_g^+$ is the eigenline of~$g$ for this eigenvalue).
In that case, $g$ admits a unique repelling projective hyperplane $X_g^-$ in $\PP(V)$, corresponding to the sum of the generalized eigenspaces for the other eigenvalues of~$g$.

This notion has been quantified as follows.
Choose a Euclidean structure on~$V$, and let $d_{\PP(V)}$ be the corresponding distance function on $\PP(V)$ given by $d_{\PP(V)}([v],[w]) = \sin \measuredangle(v,w)$ for all $v,w\in V\smallsetminus\{0\}$.
For any $\eps>0$, any point $x\in\PP(V)$, and any closed subset $X$ of $\PP(V)$, we denote by $b_x^{\eps}$ the closed ball of radius $\eps$ centered at~$x$, and by $B_X^{\eps}$ the set of points of $\PP(V)$ at distance at least $\eps$ from all points of~$X$.

\begin{defi}[{see \cite{ben97}}] \label{def:r-eps-prox-proj}
For $r\geq\eps>0$, we say that a proximal element $g\in\GL(V)$ is \emph{$(r,\eps)$-proximal in $\PP(V)$} if it satisfies the following three conditions:
\begin{enumerate}
  \item $d_{\PP(V)}(x_g^+,X_g^-)\geq 2r$,
  \item $g \cdot B_{X_g^-}^{\eps} \subset b_{x_g^+}^{\eps}$,
  \item the restriction of $g$ to $B_{X_g^-}^{\eps}$ is $\varepsilon$-Lipschitz.
\end{enumerate}
\end{defi}

Recall that a subsemigroup of $\GL(V)$ is said to act \emph{strongly irreducibly} on~$V$ if it does not preserve any finite union of linear subspaces of~$V$ apart from $\{0\}$ and~$V$.

The following celebrated result of Abels--Margulis--Soifer \cite[Th.\,4.1]{ams95} has been used extensively in recent developments around discrete subgroups of Lie groups, see \eg \cite{ben97,Benoist-notes,BenoistQuint-livre,BrGe, BS,CaTs,DGKp,DGLM,GGKW,kp22,ks,tso24} and references therein.

\begin{fact}[Abels--Margulis--Soifer] \label{fact:AMS}
Let $\Gamma$ be a semigroup and $m\geq 1$ an integer.
For any $1\leq i\leq m$, let $V_i$ be a Euclidean space of dimension $\geq 2$ and $\rho_i : \Gamma\to\GL(V_i)$ a representation such that $\rho_i(\Gamma)$ acts strongly irreducibly on~$V_i$ and contains an element which is proximal in $\PP(V_i)$.
Then there exists $r_0>0$ such that for any $r_0\geq r\geq\eps>0$, there is a finite subset $S$ of~$\Gamma$ with the following property: for any $\gamma \in \Gamma$, we can find $s\in S$ such that $\rho_i(\gamma s)$ is $(r,\eps)$-proximal in $\PP(V_i)$ for all $1\leq i\leq m$.
\end{fact}

Many applications of this result come from the fact that it allows to approach the Jordan projection by the Cartan projection.
For $V \simeq \RR^{\mathtt{d}}$, recall that the \emph{Jordan projection} (also known as the \emph{Lyapunov projection}) $\lambda : \GL(V) \to \RR^{\mathtt{d}}$ is the map sending $g$ to $(\lambda_i(g))_{1\leq i\leq\mathtt{d}}$ where $\lambda_1(g)\geq\dots\geq\lambda_{\mathtt{d}}(g)$ are the logarithms of the moduli of the complex eigenvalues of~$g$.
The \emph{Cartan projection} $\mu : \GL(V) \to \RR^{\mathtt{d}}$ is the map sending $g$ to $(\mu_i(g))_{1\leq i\leq\mathtt{d}}$ where $\mu_1(g)\geq\dots\geq\mu_{\mathtt{d}}(g)$ are the logarithms of the singular values of~$g$.

By work of Benoist \cite{ben97}, Fact~\ref{fact:AMS} implies the following (see \eg \cite[Th.\,4.12]{GGKW} for an explicit statement and proof).

\begin{cor}[Benoist] \label{cor:Benoist}
Let $V$ be a Euclidean space and let $\Gamma$ be a subsemigroup of $\GL(V)$ whose Zariski closure in $\GL(V)$ is reductive.
Then there exist a finite subset $S$ of $\Gamma$ and a constant $C>0$ such that for any $\gamma\in\Gamma$,
$$\min_{s\in S} \| \mu(\gamma) - \lambda(\gamma s) \| \leq C.$$
\end{cor}

Recall that a linear real algebraic group $G$ is said to be \emph{reductive} if the unipotent radical (\ie the largest connected unipotent normal algebraic subgroup) of the identity component of~$G$ (for the Zariski topology) is trivial.
An important subclass is that of \emph{semisimple} algebraic groups (see \cite{BorelTits,Knapp} for more background).
We note that the Zariski closure of a semigroup is a group, see \eg \cite[Lem.\,6.15]{BenoistQuint-livre}.

One important consequence of Corollary~\ref{cor:Benoist} is that the two definitions of the \emph{limit cone} of~$\Gamma$ in~$\aaa^+$ given by Benoist \cite{ben97} coincide in this setting: namely, the closure of the cone spanned by the elements $\lambda(\gamma)$ for $\gamma\in\Gamma$, and the cone spanned by the accumulation points of sequences $(\mu(\gamma_n)/\Vert\mu(\gamma_n)\Vert)_{n\in\NN}$ for sequences $(\gamma_n)\in\Gamma^{\NN}$ with $\Vert\mu(\gamma_n)\Vert \to +\infty$.
Refinements of this result, using again Corollary~\ref{cor:Benoist}, have been given \eg by Breuillard--Sert \cite{BS}.

\subsection{A simultaneous Abels--Margulis--Soifer lemma for representations into isometry groups of Gromov hyperbolic spaces} \label{subsec:intro-simult-AMS}

In this note, we consider semigroups $\Gamma$ acting by isometries on Gromov hyperbolic metric spaces $(M,d_M)$, which are not assumed to be proper nor geodesic.
The Gromov boundary $\partial_{\infty}M$ of~$M$ is then not necessarily compact.

Recall that an element $g\in\Isom(M)$ is said to be \emph{hyperbolic}, or \emph{proximal in $\partial_{\infty}M$}, if it has two fixed points $x_g^+$ and~$X_g^-$ in $\partial_{\infty}M$, with $g^n\cdot x\to x_g^+$ as $n\to +\infty$ for all $x\in\partial_{\infty}M\smallsetminus\{ X_g^-\}$.
Fix a Bourdon metric $d_{\partial_{\infty}M} = d_{a,o}$ on $\partial_{\infty}M$ (see Section~\ref{subsubsec:Bourdon-metric}).
By analogy with the linear case above, for any $\eps>0$ and any point $x\in\partial_{\infty}M$, we denote by $b_x^{\eps}$ the closed ball of radius $\eps$ centered at~$x$ and by $B_x^{\eps}$ the complement in $\partial_{\infty}M$ of the open ball of radius $\eps$ centered at~$x$, and we adopt the following terminology.

\begin{defi} \label{def:r-eps-prox-hyp}
For $r\geq\eps>0$, we say that a hyperbolic element $g\in\Isom(M)$ is \emph{$(r,\eps)$-proximal in $\partial_{\infty}M$} if it satisfies the following three conditions:
\begin{enumerate}
  \item $d_{\partial_{\infty}M}(x_g^+,X_g^-)\geq 2r$,
  \item $g \cdot B_{X_g^-}^{\eps} \subset b_{x_g^+}^{\eps}$,
  \item the restriction of $g$ to $B_{X_g^-}^{\eps}$ is $\varepsilon$-Lipschitz.
\end{enumerate}
\end{defi}

The goal of this note is to establish the following version of Fact~\ref{fact:AMS}, simultaneously for linear representations and for representations to isometry groups of Gromov hyperbolic metric spaces.

\begin{teo} \label{thm:main}
Let $\Gamma$ be a semigroup and let $I_1, I_2$ be finite sets.
For each $i\in I_1$, let $V_i$ be a Euclidean space of dimension $\geq 2$ and $\rho_i : \Gamma\to\GL(V_i)$ a representation such that $\rho_i(\Gamma)$ acts strongly irreducibly on~$V_i$ and contains an element which is proximal in $\PP(V_i)$.
For each $i\in I_2$, let $M_i$ be a Gromov hyperbolic metric space with a choice of Bourdon metric on $\partial_{\infty}M_i$, and let $\rho_i : \Gamma\to\Isom(M_i)$ be a representation such that $\rho_i(\Gamma)$ acts on $\partial_{\infty}M_i$ without a unique global fixed point and contains an element which is proximal in $\partial_{\infty}M_i$.
Then there exists $r_0>0$ such that for any $r_0\geq r\geq\eps>0$, there is a finite subset $S$ of~$\Gamma$ with the following property: for any $\gamma \in \Gamma$, we can find $s\in S$ such that $\rho_i(\gamma s)$ is $(r,\eps)$-proximal in $\PP(V_i)$ for all $i\in I_1$ and $\rho_i(\gamma s)$ is $(r,\eps)$-proximal in $\partial_{\infty}M_i$ for all $i\in I_2$.
\end{teo}

Note that in Theorem~\ref{thm:main} we do not assume $\rho_i$ to have finite kernel nor discrete image.
We also do not assume the Gromov hyperbolic metric spaces $M_i$ to be proper nor geodesic: \eg $M_i$ could be an $\R$-tree or an infinite-dimensional real hyperbolic space, see \cite{dsu17}.
Nonproper Gromov hyperbolic metric spaces have attracted increasing attention in the past few years, see \eg \cite{CCMT,DP, FiSU, Gou,MT, OR}.
One difficulty in the nonproper setting is that the Gromov boundary $\partial_{\infty}M_i$ is not compact; in particular, sequences in $\partial_{\infty}M_i$ do not necessarily converge up to passing to a subsequence.

\begin{remark} \label{remark-sideofs}
In the setting of Theorem~\ref{thm:main}, multiplication on the left can be replaced by multiplication on the right: there exists $r_0>0$ such that for any $r_0\geq r\geq\eps>0$, there is a finite subset $S$ of~$\Gamma$ with the property that for any $\gamma \in \Gamma$, we can find $s\in S$ such that $\rho_i(s \gamma)$ is $(r,\eps)$-proximal in $\PP(V_i)$ or $\partial_{\infty}M_i$ for all $i\in I_1 \cup I_2$.
Indeed, consider the formal semigroup $\Gamma^{-1}$ generated by the inverses of the elements of $\Gamma$, with the product rule $\gamma_1^{-1} \gamma_2^{-1} = (\gamma_2 \gamma_1)^{-1}$.
Each representation $\rho_i$ of~$\Gamma$ induces a representation $\rho_i^{-1}$ of $\Gamma^{-1}$ by setting $\rho_i^{-1}(\gamma^{-1}) := \rho_i(\gamma)$.
Applying Theorem~\ref{thm:main} to $\Gamma^{-1}$ and the $\rho_i^{-1}$, we obtain the existence of $r_0>0$ such that for any $r_0\geq r\geq\eps>0$, there is a finite subset $S$ of~$\Gamma$ with the property that for any $\gamma \in \Gamma$, we can find $s\in S$ such that $\rho_i^{-1}(\gamma^{-1} s^{-1}) = \rho_i(s\gamma)$ is $(r,\eps)$-proximal in $\PP(V_i)$ or $\partial_{\infty}M_i$ for all $i\in I_1 \cup I_2$.
\end{remark}

\begin{remark}
Theorem~\ref{thm:main} is not true in general if $\rho_i(\Gamma)$ admits a unique global fixed point $x_i$ in $\partial_{\infty}M_i$ for some $i \in I_2$.
Indeed, one can construct examples where some sequence $(\rho_i(\gamma_n))\in\rho_i(\Gamma)^{\NN}$ of hyperbolic elements satisfies that the fixed points $\neq x_i$ of the $\rho_i(\gamma_n)$ accumulate to~$x_i$; then, multiplying on the left by an element of a given finite set cannot uniformly separate the two fixed points of $\rho_i(\gamma_n)$.
\end{remark}

In order to prove Theorem~\ref{thm:main}, we extend the scheme of \cite{ams95}, which is designed to be applied to several linear representations at the same time, to work for certain actions on Gromov boundaries of hyperbolic metric spaces.
This requires a quantitative description of the boundary dynamics for isometries of (not necessarily proper) Gromov hyperbolic metric spaces, similar to the dynamics obtained in the linear setting from the Cartan decomposition, and which may be interesting in its own right: see Proposition~\ref{prop:Lip-hyp}.

\subsection{A simultaneous control of lengths}

In Section~\ref{sec:proof-cor-main} we use Theorem~\ref{thm:main} to prove a variant of it (Theorem~\ref{thm:prox-in-G/P}) where products of projective spaces are replaced by flag varieties of real reductive Lie groups.
We also prove the following simultaneous version of Corollary~\ref{cor:Benoist}, where we endow a product $M$ of metric spaces with a product metric $d_M$.
We fix a basepoint $o\in M$ and set, for any $g\in\Isom(M)$,
\begin{equation} \label{eqn:length-stable-length}
|g|_M := d_M(o,g\cdot o) \quad\quad\mathrm{and}\quad\quad |g|_{M,\infty} := \liminf_n \frac{|g^n|_M}{n}.
\end{equation}
Note that $|g|_{M,\infty} \leq |g|_M$ by the triangle inequality.

\begin{cor} \label{cor:main}
Let $\Gamma$ be a semigroup acting by isometries on a direct product $M$ of finitely many Gromov hyperbolic metric spaces, such that the action on none of the factors has a unique global fixed point at infinity.
For any Euclidean space~$V$ and any representation $\rho : \Gamma\to\GL(V)$ such that the Zariski closure of $\rho(\Gamma)$ in $\GL(V)$ is reductive, there exist $C>0$ and a finite subset $S$ of~$\Gamma$ with the following property: for any $\gamma \in \Gamma$ we can find $s\in S$ such that
$$\|\lambda(\rho(\gamma s)) - \mu(\rho(\gamma))\| \leq C \quad\mathrm{and}\quad | |\gamma s|_{M,\infty} - |\gamma|_M| \leq C.$$
\end{cor}

\begin{remark}
Similarly to Remark~\ref{remark-sideofs}, multiplication on the left can be replaced by multiplication on the right in Corollary~\ref{cor:main}.
\end{remark}

Theorem~\ref{thm:main} and Corollary~\ref{cor:main} apply in particular to the following situations.

\begin{ex} \label{ex:hyp-group}
A finitely generated group $\Gamma$ is Gromov hyperbolic if its Cayley graph $M$ with respect to any finite generating subset~$F$ is Gromov hyperbolic.
The action of $\Gamma$ on~$M$ is properly discontinuous, cocompact, by isometries.
The functions $|\cdot|_M$ and $|\cdot|_{M,\infty}$ are respectively the word length and stable word length with respect to~$F$.
\end{ex}

\begin{ex}\label{ex:cuspedspace}
A finitely generated group $\Gamma$ is relatively hyperbolic if it admits a cusp-uniform action by isometries on a \emph{cusped space} $M$, which is a proper Gromov hyperbolic metric space; such cusped spaces include the cusped Cayley graphs of Groves--Manning, see \cite{hh20}.
The functions $|\cdot|_M$ and $|\cdot|_{M,\infty}$ are respectively the cusped word length and stable cusped word length considered \eg in \cite{zhu21,zhu23}.
\end{ex}

We note that in \cite{DGLM} it is shown that the word length and the translation length (or stable word length) can be made to be similar in Gromov hyperbolic groups (see also \cite[Lem.\,A.2]{ZZ}).
The main point here is simultaneity (\ie finding an $s$ that works both in the Gromov hyperbolic setting and for the linear representation~$\rho$).

Corollary~\ref{cor:main} in the setting of Example~\ref{ex:hyp-group} was announced in \cite[Rem.\,4.4]{kp22}; it allows to give a somewhat different proof of a result from \cite{kp22}, namely \cite[Prop.\,1.2]{kp22}.
Since then, Corollary~\ref{cor:main} has found some applications to Anosov representations \cite{tso26} and to variants of Breuillard--Sert's joint spectrum \cite{kss}, which motivated us to write this note. 

After this work was completed, we were informed by K.~Tsouvalas that he can give a different proof of Corollary~\ref{cor:main} in the case that $\Gamma$ is a finitely generated Gromov hyperbolic group and $M$ is its Cayley graph (as in Example~\ref{ex:hyp-group}).
His proof extends to the case that $\Gamma$ is a relatively hyperbolic group acting on its Cayley graph~$M$.
This is different from our results, which concern general isometric actions on Gromov hyperbolic metric spaces and apply in particular to relatively hyperbolic groups $\Gamma$ acting on cusped spaces~$M$ as in Example~\ref{ex:cuspedspace}.

\subsection{Organization of the paper}

In Section~\ref{sec:proj} we recall some well-known results on proximality in finite-dimensional real projective spaces.
In Section~\ref{sec:Gromov-hyp} we establish some analogues in boundaries of Gromov hyperbolic spaces.
In Section~\ref{s.movingpoints} we discuss how to create transversality by applying finitely many transformations in the setting of Theorem~\ref{thm:main}.
In Section~\ref{sec:simult-AMS-simult-prox} we prove Theorem~\ref{thm:main} under the assumption that there exists a simultaneously proximal element.
We then show in Section~\ref{sec:proof-exists-simult-prox} that this assumption is always satisfied.
Finally, in Section~\ref{sec:proof-cor-main} we prove a variant of Theorem~\ref{thm:main} in flag varieties of real reductive Lie groups, as well as Corollary~\ref{cor:main}.

\subsection*{Acknowledgments}

We would like to thank Matias Carrasco, Pablo Lessa and Konstantinos Tsouvalas for useful comments and encouragement.

\section{The linear setting} \label{sec:proj}

In this section we collect some preliminary results in the classical setting of Fact~\ref{fact:AMS} and Corollary~\ref{cor:Benoist}.

In the whole section, we fix a finite-dimensional real vector space~$V$, of dimension $\mathtt{d}\geq 2$.
We choose a Euclidean structure on~$V$ and endow the projective space $\PP(V)$ with the corresponding distance function $d_{\PP(V)}$ given by $d_{\PP(V)}([v],[w]) = \sin \measuredangle(v,w)$ for all $v,w\in V\smallsetminus\{0\}$.
If $X,Y$ are closed subsets of $\PP(V)$, we denote by $d_{\PP(V)}(X,Y)$ their Hausdorff distance.

\subsection{A uniform contraction property} \label{subsec:unif-contract-proj}

As in Section~\ref{subsec:intro-classical-AMS}, for $g\in\GL(V)$ we denote by $\mu_1(g)\geq\dots\geq\mu_{\mathtt{d}}(g)$ the logarithms of the singular values of~$g$; we set $(\mu_1-\mu_2)(g) :=\linebreak \mu_1(g) - \mu_2(g)$.
The following is a refinement of \cite[Lem.\,5.6]{Benoist-notes}.

\begin{prop} \label{prop:Lip-proj}
To any $g\in\GL(V)$ we can associate a point $y_g^+\in\PP(V)$ and a projective hyperplane $Y_g^-$ of $\PP(V)$ with the following property: for any $\eps>0$, there exists $D \geq 1$ such that for any $g\in\GL(V)$, the restriction of $g$ to $B^{\eps}_{Y_g^-}$ is $D e^{-(\mu_1-\mu_2)(g)}$-Lipschitz and takes values in a ball of radius $\leq D e^{-(\mu_1-\mu_2)(g)}$ centered at~$y_g^+$ for $d_{\PP(V)}$.
\end{prop}

\begin{proof}
The group $\GL(V)$ admits the Cartan decomposition $\GL(V) = K \exp(\aaa^+) K$ where $K \simeq \mathrm{O}(\mathtt{d})$ is the maximal compact subgroup of $\GL(V)$ preserving our chosen Euclidean structure on $V \simeq \RR^{\mathtt{d}}$, and $\aaa^+$ is the set of real diagonal matrices of size $\mathtt{d}\times\mathtt{d}$ with entries in nonincreasing order.
Any element $g\in\GL(V)$ can be written as $g = k\exp(a)k'$ for some $k,k'\in K$ and a unique $a\in\aaa^+$, whose diagonal entries are $\mu_1(g),\dots,\mu_{\mathtt{d}}(g)$.
For each~$g$ we choose such $k,k'\in K$ and set
$$y_g^+ := k\cdot [e_1] \quad\quad\mathrm{and}\quad\quad Y_g^- := {k'}^{-1}\cdot\PP(\mathrm{span}(e_2,\dots,e_{\mathtt{d}})),$$
where $(e_1,\dots,e_{\mathtt{d}})$ is the canonical basis of~$\R^{\mathtt{d}}$.
Let $y^+ := [e_1]$ and $Y^- := \PP(\mathrm{span}(e_2,\dots,e_{\mathtt{d}}))$.
Since $K$ acts on $\PP(V)$ by isometries for $d_{\PP(V)}$, it is sufficient to prove that for any $\eps>0$, there exists $D \geq 1$ such that for any $a = \mathrm{diag}(a_1,\dots,a_{\mathtt{d}})\in\aaa^+$, the restriction of $\exp(a)$ to $B^{\eps}_{Y^-}$ is $D e^{-(a_1-a_2)}$-Lipschitz and takes values in a ball of radius $\leq D e^{-(a_1-a_2)}$ centered at~$y^+$.

The map $\varphi : \mathrm{span}(e_2,\dots,e_{\mathtt{d}}) \to \PP(V)\smallsetminus Y^-$ sending $v$ to $[e_1+v]$ is a homeomorphism, and $\varphi$ and $\varphi^{-1}$ are both Lipschitz on any compact set.
With this identification, for any $a = \mathrm{diag}(a_1,\dots,a_{\mathtt{d}})\in\aaa^+$, the induced action of $\exp(a)$ on $\mathrm{span}(e_2,\dots,e_{\mathtt{d}})$ is given by
$$(\varphi^{-1}\circ\exp(a)\circ\varphi)(x_2 e_2 + \dots + x_{\mathtt{d}} e_{\mathtt{d}}) = e^{a_2-a_1} x_2 e_2 + \dots + e^{a_{\mathtt{d}}-a_1} x_{\mathtt{d}} e_{\mathtt{d}}.$$
This is $e^{-(a_1-a_2)}$-Lipschitz on $\mathrm{span}(e_2,\dots,e_{\mathtt{d}})$, hence for any $\eps>0$ the action of $\exp(a)$ on $B^{\eps}_{Y^-}$ is $D e^{-(a_1-a_2)}$-Lipschitz where
$$D = \mathrm{Lip}\Big(\varphi^{-1}|_{B^{\eps}_{Y^-}}\Big) \cdot \mathrm{Lip}\Big(\varphi|_{\varphi^{-1}(B^{\eps}_{Y^-})}\Big).$$
Since $\exp(a)$ fixes~$y^+$, since the restriction of $\exp(a)$ to $B^{\eps}_{Y^-}$ is $D e^{-(a_1-a_2)}$-Lipschitz, and since any point of $\PP(V)$ is at distance $\leq 1$ from~$y^+$ for $d_{\PP(V)}$, we deduce that $\exp(a)(B^{\eps}_{Y^-})$ is contained in the ball of radius $D e^{-(a_1-a_2)}$ centered at~$y^+$.
\end{proof}

We note that if $g$ is proximal in $\PP(V)$, then $y_{g^n}^+ \to x_g^+$ and  $Y_{g^n}^- \to X_g^-$ as $n\to +\infty$: see \eg \cite[Lem.\,5.11]{GGKW}.
More generally, we refer to \cite[\S\,5]{GGKW} or \cite[\S A]{BPS} for more convergence properties, to be compared with Proposition~\ref{prop:Lip-hyp}.\eqref{item:Lip-hyp-2} below.

\subsection{A sufficient condition for $(r,\eps)$-proximality} \label{subsec:suff-prox-proj}

Recall the notion of $(r,\eps)$-proximality in $\PP(V)$ from Definition~\ref{def:r-eps-prox-proj}.
We note that if $g$ is $(r,\eps)$-proximal in $\PP(V)$, then it is also $(r',\eps')$-proximal in $\PP(V)$ for all $r\geq r'\geq\eps'\geq\eps>0$; moreover, $r$ must be $<1$ since $d_{\PP(V)}$ is bounded above by~$1$.

The following is \cite[Lem.\,5.2]{Benoist-notes}; we give a brief proof for the reader's convenience.

\begin{lema} \label{l.critprox-proj}
Let $r\geq \eps>0$.
For $g\in\GL(V)$, if there exist $y^+\in\PP(V)$ and a projective hyperplane $Y^-$ of~$\PP(V)$ such that
\begin{enumerate}
  \item $d_{\PP(V)}(y^+,Y^-) \geq 6r$,
  \item $g \cdot B_{Y^-}^{\eps} \subset b_{y^+}^{\eps}$,
  \item the restriction of $g$ to $B_{Y^-}^{\eps}$ is $\varepsilon$-Lipschitz, 
\end{enumerate}
then $g$ is $(2r,2\eps)$-proximal with $d_{\PP(V)}(x^+_g,y^+)\leq \eps$ and $d_{\PP(V)}(X_g^-,Y^-) \leq \eps$.
\end{lema}

\begin{proof}
The element $g$ sends $b_{y^+}^{\eps}$ inside itself by (1) and~(2), and the restriction of $g$ to~$b_{y^+}^{\eps}$ is $\eps$-Lipschitz by~(3).
Therefore, $g$ has an attracting fixed point inside~$b_{y^+}^{\eps}$, and so $g$ is proximal in $\PP(V)$ with $d_{\PP(V)}(x_g^+,y^+)\leq\eps$.
Since points of $X_g^-$ are \emph{not} attracted towards $x_g^+$ by~$g$, the set $X_g^-$ does not meet $B_{Y^-}^{\eps}$ by~(3).
Since $X_g^-$ and~$Y^-$ are both projective hyperplanes, we deduce $d_{\PP(V)}(X_g^-,Y^-) \leq \eps$.
Thus, using the triangle inequality, we have $d_{\PP(V)}(x_g^+,X_g^-) \geq d_{\PP(V)}(y^+,Y^-) - 2\eps \geq 4r$, and $B_{X_g^-}^{2\eps} \subset B_{Y^-}^{\eps}$, and $g\cdot B_{X_g^-}^{2\eps} \subset b_{x_g^+}^{2\eps}$.
Moreover, the restriction of $g$ to $B_{X_g^-}^{2\eps}$ is $\eps$-Lipschitz, hence also $2\eps$-Lipschitz.
\end{proof}

\subsection{A consequence of $(r,\eps)$-proximality} \label{subsec:prox-conseq-proj}

As in Section~\ref{subsec:intro-classical-AMS}, for $g\in\GL(V)$ we denote by $\lambda_1(g)\geq\dots\geq\lambda_{\mathtt{d}}(g)$ the logarithms of the moduli of the complex eigenvalues of~$g$.
Note that $e^{\lambda_1(g)}$ is the spectral radius of~$g$, while $e^{\mu_1(g)}$ is the operator norm of~$g$.
Therefore, $\lambda_1(g)\leq\mu_1(g)$ and
\begin{equation} \label{eqn:triangle-ineq-mu1}
\mu_1(gg') \leq \mu_1(g) + \mu_1(g')
\end{equation}
for all $g,g'\in\GL(V)$.
We have
\begin{equation} \label{eqn:lambda-mu}
\lambda_i(g) = \lim_{n\to +\infty} \frac{1}{n}\,\mu_i(g^n)
\end{equation}
for all $1\leq i\leq\mathtt{d}$ and all $g\in\GL(V)$, and $g$ is proximal in $\PP(V)$ if and only if $\lambda_1(g) > \lambda_2(g)$.

\begin{lema} \label{l.approachgap-proj}
For any element $g\in\GL(V)$ which is proximal in $\PP(V)$,
$$\mu_1(g) - \bigg|\log\frac{d_{\PP(V)}(x_g^+,X_g^-)^2}{2}\bigg| \leq \lambda_1(g) \leq \mu_1(g). $$
\end{lema}

\begin{proof}
Let $g\in\GL(V)$ be proximal in $\PP(V)$. 
Let $(e_1,\dots,e_{\mathtt{d}})$ be an orthonormal basis of $V \simeq \RR^{\mathtt{d}}$ such that $X_g^- = \PP(\mathrm{span}(e_2,\dots,e_{\mathtt{d}}))$.
We can write $x_g^+ = u\cdot [e_1]$ where $u\in\GL(V)$ is a unipotent element whose restriction to $\mathrm{span}(e_2,\dots,e_{\mathtt{d}})$ is the identity, and $u\cdot e_1 = e_1 + v$ for some $v \in \mathrm{span}(e_2,\dots,e_{\mathtt{d}})$.
Then $h := u^{-1}gu$ is proximal in $\PP(V)$ with $x_h^+ = [e_1]$ and $X_h^- = \PP(\mathrm{span}(e_2,\dots,e_{\mathtt{d}}))$, hence $\mu_1(h) = \lambda_1(h) = \lambda_1(g)$.
Using \eqref{eqn:triangle-ineq-mu1}, we deduce
$$|\mu_1(g) - \lambda_1(g)| = |\mu_1(g) - \mu_1(h)| \leq \mu_1(u) + \mu_1(u^{-1}).$$

An elementary computation shows that the characteristic polynomial of ${}^t\!uu$ is\linebreak $(X-1)^{n-2} (X^2 - (2+\Vert v\Vert^2)X + 1)$, hence
$$\mu_1(u) = \mu_1(u^{-1}) = \frac{1}{2} \log\!\bigg(1 + \frac{\Vert v\Vert^2 + \sqrt{\Vert v\Vert^2 (4 + \Vert v\Vert^2)}}{2}\bigg).$$
We deduce $\mu_1(u) + \mu_1(u^{-1}) \leq \log((1+\Vert v\Vert)^2) \leq \log(2(1+\Vert v\Vert^2))$.

Finally, by definition $d_{\PP(V)}(x_g^+,X_g^-)$ is the sine of the angle between $u\cdot e_1 = e_1 + v$ and $\mathrm{span}(e_2,\dots,e_{\mathtt{d}})$, hence $d_{\PP(V)}(x_g^+,X_g^-) = 1/\sqrt{1+\Vert v\Vert^2}$.
Thus we obtain
$$|\mu_1(g) - \lambda_1(g)| \leq \log(2(1+\Vert v\Vert^2)) = \bigg|\log\frac{d_{\PP(V)}(x_g^+,X_g^-)^2}{2}\bigg|. \qedhere$$
\end{proof}

\begin{cor} \label{cor:r-eps-prox-lambda-mu}
For any $r\geq\varepsilon>0$ and any $g\in\GL(V)$, if $g$ is $(r,\varepsilon)$-proximal in $\PP(V)$, then
$$\mu_1(g) - |\log 2r^2| \leq \lambda_1(g) \leq \mu_1(g). $$
\end{cor}

\section{The Gromov hyperbolic setting} \label{sec:Gromov-hyp}

We now establish analogues of the results of Section~\ref{sec:proj} for groups of isometries of Gromov hyperbolic metric spaces.
We first introduce some notation.

\subsection{Preliminaries: Gromov hyperbolic metric spaces}\label{subsec:GH}

In the whole section, we consider a metric space $(M,d_M)$.

\subsubsection{Gromov product}

For any $m,p,q\in M$, we denote by
\begin{equation} \label{eq:Gromovproduct}
(m|p)_q = \frac{1}{2} (d_M(m,q) + d_M(p,q) - d_M(m,p))
\end{equation} 
the \emph{Gromov product} between $m$ and~$p$ based at~$q$. 
We assume that $(M,d_M)$ is \emph{Gromov hyperbolic} in sense that there exist $\delta>0$ and $o\in M$ such that any $m,p,q\in M$ satisfy the \emph{Gromov inequality}
\begin{equation} \label{eqn:Gromov-ineq}
(m|p)_o \geq \min\big( (m|q)_o, (p|q)_o\big) - \delta.
\end{equation}
If the metric space $(M,d_M)$ is proper and geodesic, then $(m|p)_q$ is comparable up to a uniform constant to the distance from $q$ to a geodesic joining $m$ and $p$ (see \cite[Def.\,III.H.1.9]{BH}, or it can also be deduced by approximation by trees, see \eg \cite[Th.\,1.1]{Coo}).
However, in this note we do \emph{not} assume $(M,d_M)$ to be proper or geodesic.

\subsubsection{Gromov boundary}

A sequence $(m_n)\in M^{\NN}$ is called a \emph{Gromov sequence} if $(m_n|m_m)_o \underset{n,m}{\longrightarrow} +\infty$.
The \emph{Gromov boundary} of $(M,d_M)$ is the set $\partial_{\infty}M$ of equivalence classes of Gromov sequences for the equivalence relation $\sim$ given by $(m_n)_{n\in\NN} \sim (p_n)_{n\in\NN}$ if $(m_n|p_n)_o \to +\infty$, or equivalently if $(m_n|p_m)_o \underset{n,m}{\longrightarrow} +\infty$.
We note that the notions of Gromov sequence and equivalence do not depend on the basepoint $o\in M$ (see \cite[Prop.\,3.3.3.(d)]{dsu17}).
In particular, the Gromov boundary $\partial_{\infty}M$ is independent of~$o$.
If $(M,d_M)$ is proper and geodesic, then $\partial_{\infty}M$ coincides with the set of equivalence classes of geodesic rays under asymptotic equivalence (see \eg \cite[Lem.\,III.H.3.13]{BH}).

\subsubsection{Topology on the Gromov bordification}

We can extend the Gromov product from $M$ to $M\cup\partial_{\infty} M$ by setting
$$(x_1|x_2)_o = \inf \liminf_n (m_{n,1}|m_{n,2})_o$$
for all $x_1,x_2\in M\cup\partial_{\infty} M$, where the infimum is taken among all sequences $(m_{n,i})\in M^{\NN}$ in the equivalence class $x_i$ if $x_i\in\partial_{\infty}M$, or $(m_{n,i})\in M^{\NN}$ constant equal to $x_i$ if $x_i\in M$ (see \cite[Def.\,III.H.3.15]{BH} and \cite[Def.\,3.4.9 \& Lem.\,3.4.10]{dsu17}). Note that this definition has slight changes in the literature, but this is well defined up to bounded errors (see \cite[Lem.\,3.4.7 \& Lem.\,3.4.10]{dsu17}).

Then (see \eg \cite[Cor.\,3.4.12]{dsu17}) there exists $C>0$ such that for any $x\in\partial_{\infty}M$ and any $m\in M$,
\begin{equation} \label{eqn:dist-and-Gromov-prod}
\big| (o|x)_m + (m|x)_o - d_M(o,m) \big| \leq C.
\end{equation}

We endow $M\cup\partial_{\infty}M$ with its natural topology, for which open subsets of $M\cup\partial_{\infty}M$ meet $M$ in open subsets, and a basis of neighborhoods of any $x \in \partial_{\infty}M$ is given by the sets $\mathcal{N}_t(x)$ of points $m \in M \cup \partial_\infty M$ such that $(m|x)_o > t$ (see \cite[\S\,3.4.2]{dsu17}).

If $(M,d_M)$ is proper and geodesic, then $M\cup\partial_{\infty}M$ (hence $\partial_{\infty}M$) is compact; we then call  $M\cup\partial_{\infty}M$ the \emph{Gromov compactification} of~$M$.
However, in general $M\cup\partial_{\infty}M$ and $\partial_{\infty}M$  are not compact when $(M,d_M)$ is not proper; we call $M\cup\partial_{\infty}M$ the \emph{Gromov bordification} of~$M$.

\subsubsection{Bourdon metric} \label{subsubsec:Bourdon-metric}

For $a>1$ and $m,p\in M$, we set
$$d_{a,o}(m,p) := \inf \sum_{n=0}^{N-1} a^{-(m_n|m_{n+1})_o},$$
where the infimum is taken over all finite sequences $(m_n)_{0\leq n\leq N}$ in~$M$ with $m_0 = m$ and $m_N = p$.
If $a>1$ is close enough to~$1$, then $d_{a,o} : M\times M\to [0,+\infty)$ extends by continuity to a function $d_{a,o} : (M\cup\partial_{\infty} M) \times (M\cup\partial_{\infty} M) \to [0,+\infty)$ satisfying
\begin{equation} \label{eqn:Bourdon-metric-estim}
\frac{1}{4} \, a^{-(x_1|x_2)_o} \leq d_{a,o}(x_1,x_2) \leq a^{-(x_1|x_2)_o}
\end{equation}
for all $x_1,x_2$; the restriction of $d_{a,o}$ to $\partial_{\infty}M$ is a complete metric called the \emph{Bourdon metric} with parameters $a$ and~$o$ (see \cite[Prop.\,3.6.8]{dsu17}).

\subsubsection{Isometries} \label{subsubsec:isom-hyp}

We denote by $\Isom(M)$ the group of isometries of the Gromov hyperbolic metric space $(M,d_M)$.
Recall (see \eg \cite[Th.\,6.1.4]{dsu17}) that any nontrivial element $g\in\Isom(M)$ is either elliptic, parabolic, or hyperbolic, where we say that $g$ is
\begin{itemize}
  \item \emph{elliptic} if the orbit $(g^n\cdot o)_{n\in\NN}$ is bounded (in that case any orbit $(g^n\cdot m)_{n\in\NN}$ with $m\in M$ is bounded by the triangle inequality, and if $(M,d_M)$ is CAT(0) this implies that $g$ has a fixed point in~$M$),
  \item \emph{parabolic} if $g$ is not elliptic and has a unique fixed point in $\partial_{\infty}M$,
  \item \emph{hyperbolic}\footnote{The terminology used in \cite{dsu17} is \emph{loxodromic}.}, or \emph{proximal in $\partial_{\infty}M$}, if $g$ has two fixed points $x_g^+$ and~$x_g^-$ in $\partial_{\infty}M$, with $g^n\cdot m\to x_g^{\pm}$ for all $m\in M\cup\partial_{\infty}M\smallsetminus\{ x_g^{\mp}\}$ as $n\to\pm\infty$.
\end{itemize}
By Gromov's classification (see \eg \cite[Th.\,6.2.3 \& Prop.\,6.2.14]{dsu17}), any subsemigroup $\Gamma$ of $\Isom(M)$ is one of the following:
\begin{itemize}
  \item \emph{elliptic} (\ie the orbit $\Gamma\cdot o$ is bounded in~$M$; as above this implies that $\Gamma\cdot m$ is bounded for any $m\in M$),
  \item \emph{parabolic} (\ie it is not elliptic, it has a global fixed point in $\partial_{\infty}M$, and it does not contain any hyperbolic isometry of~$M$),
  \item \emph{focal} (\ie it contains a hyperbolic isometry and it has a unique global fixed point in $\partial_{\infty}M$),
  \item \emph{lineal} (\ie it contains a hyperbolic isometry and all hyperbolic isometries in~$\Gamma$ have the same two fixed points in $\partial_{\infty}M$),
  \item \emph{of general type} (\ie it contains two hyperbolic isometries whose fixed points are totally distinct).
\end{itemize}
If $\Gamma$ is lineal and if $\{ x_0,x_1\}$ is the set of fixed points of the proximal elements of~$\Gamma$, we shall denote by $\Gamma_{\mathrm{fix}}$ the subsemigroup of~$\Gamma$ consisting of those elements fixing $x_0$ and~$x_1$ pointwise; it includes all elements of~$\Gamma$ which are proximal in $\partial_{\infty}M$.

\subsection{A uniform contraction property}

As in Section~\ref{subsec:intro-simult-AMS}, for any $\eps>0$ and any $x\in\partial_{\infty}M$, we denote by $b_{x}^{\eps}$ the closed ball of radius $\eps$ centered at~$x$ and by $B_{x}^{\eps}$ the complement of the open ball of radius $\eps$ centered at~$x$ for $d_{\partial_{\infty}M}$.
Here is an analogue of Proposition~\ref{prop:Lip-proj}.

\begin{prop} \label{prop:Lip-hyp}
Let $(M,d_M)$ be a Gromov hyperbolic metric space.
We endow the Gromov boundary $\partial_\infty M$ with a Bourdon metric $d_{\partial_{\infty}M} = d_{a,o}$ as in Section~\ref{subsubsec:Bourdon-metric}.
Fix $\eps>0$.
Then we can associate to any $g\in\Isom(M)$ some points $y_g^+,Y_g^-\in\partial_{\infty}M$ so that the following hold:
\begin{enumerate}
  \item\label{item:Lip-hyp-1} there exists $D\geq 1$ such that for any $g\in\Isom(M)$, the restriction of $g$ to $B^{\eps}_{Y_g^-}$ is $D a^{-d_M(o,g\cdot o)}$-Lipschitz and takes values in a ball of radius $\leq D a^{-d_M(o,g\cdot o)}$ centered at $y_g^+$ for~$d_{\partial_{\infty}M}$;
  \item\label{item:Lip-hyp-2} for any $(g_n)\in\Isom(M)^{\NN}$ and any $x^+\in\partial_{\infty}M$, we have $g_n\cdot o\to x^+$ if and only if $d_M(o,g_n\cdot o)\to +\infty$ and $y_{g_n}^+ \to x^+$;
  \item\label{item:Lip-hyp-3} for any $(g_n)\in\Isom(M)^{\NN}$ and any $x^-\in\partial_{\infty}M$, we have $g_n^{-1}\cdot o\to x^-$ if and only if $d_M(o,g_n\cdot o)\to +\infty$ and $Y_{g_n}^- \to x^-$.
\end{enumerate}
\end{prop}

In particular, \eqref{item:Lip-hyp-2} and~\eqref{item:Lip-hyp-3} imply that for any hyperbolic element $g\in\Isom(M)$ we have $y_{g^n}^+ \to x_g^+$ and $Y_{g^n}^- \to X_g^-$ as $n\to +\infty$.

In order to prove Proposition~\ref{prop:Lip-hyp}, for $m,p\in M$ and $\sigma>0$ we consider, as in \cite[Def.\,4.5.1]{dsu17}, the \emph{shadow cast by $p$ from the light source $m$ with parameter~$\sigma$}, which is defined to be
$$\Shad_m(p,\sigma) := \{ x\in\partial_{\infty}M ~|~ (m|x)_p \leq \sigma\}.$$
It is a nonempty closed subset of $\partial_{\infty}M$ (see \cite[Obs.\,4.5.2]{dsu17}).

\begin{lema} \label{lem:Shad}
Let $(M,d_M)$ be a Gromov hyperbolic metric space.
We endow the Gromov boundary $\partial_\infty M$ with a Bourdon metric $d_{\partial_{\infty}M} = d_{a,o}$ as in Section~\ref{subsubsec:Bourdon-metric}.
Let $\delta,C>0$ be given by \eqref{eqn:Gromov-ineq} and \eqref{eqn:dist-and-Gromov-prod}.
Then for any $m,p\in M$ and any $\sigma>0$,
\begin{enumerate}[(i)]
  \item\label{item:Shad-1} (see \cite[Lem.\,4.5.7]{dsu17}) $\partial_{\infty}M \smallsetminus \Shad_m(o,\sigma)$ has diameter $\leq a^{\delta-\sigma}$;
  \item\label{item:Shad-2} if $\sigma \geq d_M(m,p) + C$, then $\Shad_m(p,\sigma) = \partial_{\infty}M$;
  \item\label{item:Shad-3} (see \cite[Lem.\,4.5.8]{dsu17}) $\Shad_o(p,\sigma)$ has diameter $\leq a^{C+\delta+\sigma-d_M(o,p)}$;
  \item\label{item:Shad-4} if $d_M(o,p) > 2\sigma + C$, then $\Shad_o(p,\sigma) \subset \partial_{\infty}M \smallsetminus \Shad_p(o,\sigma)$.
\end{enumerate}
\end{lema}

\begin{proof}[Proof of Lemma~\ref{lem:Shad}]
\eqref{item:Shad-1} Let $x_1,x_2 \in \partial_{\infty}M \smallsetminus \Shad_m(o,\sigma)$, \ie $(m|x_i)_o > \sigma$ for all $i\in\{1,2\}$.
By the Gromov inequality \eqref{eqn:Gromov-ineq} we have $(x_1|x_2)_o \geq \min_i (m|x_i)_o - \delta > \sigma - \delta$, hence\linebreak $d_{\partial_{\infty}M}(x_1,x_2) <\nolinebreak a^{\delta-\sigma}$ by \eqref{eqn:Bourdon-metric-estim}.

\eqref{item:Shad-2} By \eqref{eqn:dist-and-Gromov-prod}, we have $(m|x)_p \leq d_M(m,p) + C$ for all $x\in\partial_{\infty}M$.

\eqref{item:Shad-3} Let $x_1,x_2 \in \Shad_o(p,\sigma)$, \ie $(o|x_i)_p \leq \sigma$ for all $i\in\{1,2\}$.
By \eqref{eqn:dist-and-Gromov-prod} we have $(p|x_i)_o \geq d_M(o,p) - \sigma - C$.
The Gromov inequality \eqref{eqn:Gromov-ineq} then yields $(x_1|x_2)_o \geq \min_i (p|x_i)_o - \delta > d_M(o,p) - \sigma - C - \delta$, hence $d_{\partial_{\infty}M}(x_1,x_2) < a^{C+\delta+\sigma-d_M(o,p)}$ by \eqref{eqn:Bourdon-metric-estim}.

\eqref{item:Shad-4} Let $x \in \Shad_o(p,\sigma)$, \ie $(o|x)_p \leq \sigma$.
By \eqref{eqn:dist-and-Gromov-prod} we have $(p|x)_o \geq d_M(o,p) - \sigma - C$.
Therefore, if $d_M(o,p) > 2\sigma + C$, then $(p|x)_o > \sigma$, \ie $x \in \partial_{\infty}M \smallsetminus \Shad_p(o,\sigma)$.
\end{proof}

By $\Isom(M)$-invariance of the Gromov product, $g\cdot\nolinebreak\Shad_{g^{-1}\cdot o}(o,\sigma) = \Shad_o(g\cdot o,\sigma)$ for all $g\in\Isom(M)$.

\begin{prop} \label{prop:Lip-hyp-Shad}
Let $(M,d_M)$ be a Gromov hyperbolic metric space.
We endow the Gromov boundary $\partial_\infty M$ with a Bourdon metric $d_{\partial_{\infty}M} = d_{a,o}$ as in Section~\ref{subsubsec:Bourdon-metric}.
Then
\begin{enumerate}
  \item\label{item:Lip-hyp-Shad-1} there exists $D'\geq 1$ such that for any $\sigma>0$ and any $g\in\Isom(M)$, the set $\partial_{\infty}M \smallsetminus \Shad_{g^{-1}\cdot o}(o,\sigma)$ has diameter $\leq D' a^{-\sigma}$, the restriction of $g$ to $\Shad_{g^{-1}\cdot o}(o,\sigma)$ is\linebreak $D' a^{2\sigma-d_M(o,g\cdot o)}$-Lipschitz, and its image $g\cdot\Shad_{g^{-1}\cdot o}(o,\sigma) = \Shad_o(g\cdot o,\sigma)$ has diameter $\leq D' a^{\sigma-d_M(o,g\cdot o)}$ for~$d_{\partial_{\infty}M}$;
  \item\label{item:Lip-hyp-Shad-2} for any $\sigma>0$, any $x^+\in\partial_{\infty}M$, any $(g_n)\in\Isom(M)^{\NN}$, and any $(x_n)\in (\partial_{\infty}M)^{\NN}$ with $x_n\in\Shad_o(g_n\cdot o,\sigma)$ for all~$n$, we have $g_n\cdot o\to x^+$ if and only if $d_M(o,g_n\cdot o)\to +\infty$ and $x_n \to x^+$.
\end{enumerate}
\end{prop}

\begin{proof}[Proof of Proposition~\ref{prop:Lip-hyp-Shad}]
\eqref{item:Lip-hyp-Shad-1} By \cite[Lem.\,4.5.6]{dsu17} and its proof, the restriction of $g$ to\linebreak $\Shad_{g^{-1}\cdot o}(o,\sigma)$ is $C' a^{2\sigma-d_M(o,g\cdot o)}$-Lipschitz for some constant $C'>0$ independent of~$\sigma$.
By Lemma~\ref{lem:Shad}.\eqref{item:Shad-1}--\eqref{item:Shad-3}, we can then take $D' := \max(C',a^{C+\delta})$, where $\delta,C>0$ are given by \eqref{eqn:Gromov-ineq} and \eqref{eqn:dist-and-Gromov-prod}.

\eqref{item:Lip-hyp-Shad-2} We first note that $g_n\cdot o \to x^+$ is equivalent to $(g_n\cdot o|x^+)_o \to +\infty$.
If this is satisfied, then by \eqref{eqn:dist-and-Gromov-prod} we have $d_M(o,g_n\cdot o) \to +\infty$.

We now assume $d_M(o,g_n\cdot o) \to +\infty$.
For any $n$ we have $(o|x_n)_{g_n\cdot o} \leq \sigma$, hence $(g_n\cdot o|x_n)_o \to +\infty$ by \eqref{eqn:dist-and-Gromov-prod}.
Let us check that $(g_n\cdot o|x^+)_o \to +\infty$ if and only if $(x_n|x^+)_o \to +\infty$.

On the one hand, by the Gromov inequality \eqref{eqn:Gromov-ineq} we have $$(x_n|x^+)_o \geq \min( (g_n\cdot o|x_n)_o, (g_n\cdot o|x^+)_o) - \delta.$$
\noindent Therefore, if $(g_n\cdot o|x^+)_o \to +\infty$, then $(x_n|x^+)_o \to +\infty$.

On the other hand, by the Gromov inequality \eqref{eqn:Gromov-ineq} we have $$(g_n\cdot o|x^+)_o \geq \min( (g_n\cdot o|x_n)_o, (x^+|x_n)_o) - \delta.$$
\noindent Therefore, if $(x^+|x_n)_o \to +\infty$, then $(g_n\cdot o|x^+)_o \to +\infty$.
\end{proof}

\begin{proof}[Proof of Proposition~\ref{prop:Lip-hyp}]
Let $C,D'>0$ be given respectively by \eqref{eqn:dist-and-Gromov-prod} and Proposition~\ref{prop:Lip-hyp-Shad}.\eqref{item:Lip-hyp-Shad-1}.
Choose $\sigma>0$ large enough so that $D' a^{-\sigma} \leq \eps$.
By Proposition~\ref{prop:Lip-hyp-Shad}.\eqref{item:Lip-hyp-Shad-1}, the set $\partial_{\infty}M \smallsetminus \Shad_{g^{-1}\cdot o}(o,\sigma)$ has diameter $\leq\eps$ for all $g\in\Isom(M)$, and so we can choose $Y_g^- \in \partial_{\infty}M$ so that $\partial_{\infty}M \smallsetminus \Shad_{g^{-1}\cdot o}(o,\sigma) \subset b^{\eps}_{Y_g^-}$; equivalently, $B^{\eps}_{Y_g^-} \subset \Shad_{g^{-1}\cdot o}(o,\sigma)$.
In fact, by Lemma~\ref{lem:Shad}.\eqref{item:Shad-4}, for any $g\in\Isom(M)$ with $d_M(o,g\cdot o) > 2\sigma + C$ we have $\Shad_o(g^{-1}\cdot o,\sigma) \subset \partial_{\infty}M \smallsetminus \Shad_{g^{-1}\cdot o}(o,\sigma)$, and so we may and do choose $Y_g^- \in \Shad_o(g^{-1}\cdot o,\sigma)$ in that case.
On the other hand, we choose any $y_g^+ \in \Shad_o(g\cdot o,\sigma)$.

\eqref{item:Lip-hyp-1} Set $D := D' a^{2\sigma}$.
By Proposition~\ref{prop:Lip-hyp-Shad}.\eqref{item:Lip-hyp-Shad-1}, the restriction of $g$ to $B^{\eps}_{Y_g^-} \subset \Shad_{g^{-1}\cdot o}(o,\sigma)$ is $D a^{-d_M(o,g\cdot o)}$-Lipschitz, and the image $g\cdot\Shad_{g^{-1}\cdot o}(o,\sigma) = \Shad_o(g\cdot o,\sigma)$ has diameter $\leq D a^{-d_M(o,g\cdot o)}$.
In particular, $\Shad_o(g\cdot o,\sigma)$ is contained in the ball of radius $\leq D a^{-d_M(o,g\cdot o)}$ centered at~$y_g^+$.

\eqref{item:Lip-hyp-2} This follows immediately from Proposition~\ref{prop:Lip-hyp-Shad}.\eqref{item:Lip-hyp-Shad-2}, taking $x_n = y_{g_n}^+$.

\eqref{item:Lip-hyp-3} This follows immediately from Proposition~\ref{prop:Lip-hyp-Shad}.\eqref{item:Lip-hyp-Shad-2} applied to $g_n^{-1}$ instead of~$g_n$, taking $x_n = Y_{g_n}^-$, which by construction belongs to $\Shad_o(g_n^{-1}\cdot o,\sigma)$ for all large enough~$n$.
\end{proof}

\subsection{A sufficient condition for $(r,\eps)$-proximality} \label{subsec:suff-prox-hyp}

Recall the notion of $(r,\eps)$-proximality in $\partial_{\infty}M$ from Definition~\ref{def:r-eps-prox-hyp}.

\begin{remark} \label{rem:r-eps-r'-eps'-prox-hyp}
If $g\in\Isom(M)$ is $(r,\eps)$-proximal in $\partial_{\infty}M$, then it is also $(r',\eps')$-proximal in $\partial_{\infty}M$ for all $r\geq r'\geq\eps'\geq\eps>0$; moreover, $r$ must be $<1$ since $d_{\partial_{\infty}M}$ is bounded above by~$1$.
\end{remark}

The following holds with a similar proof to Lemma~\ref{l.critprox-proj}.

\begin{lema} \label{l.critprox-hyp}
Let $(M,d_M)$ be a Gromov hyperbolic metric space.
We endow the Gromov boundary $\partial_\infty M$ with a Bourdon metric $d_{\partial_{\infty}M} = d_{a,o}$ as in Section~\ref{subsubsec:Bourdon-metric}.
Let $r\geq \eps>0$.
For $g\in\Isom(M)$, if there exist $w^+,w^- \in \partial_{\infty}M$ such that
\begin{enumerate}
  \item $d_{\partial_{\infty}M}(w^+,w^-) \geq 6r$,
  \item $g \cdot B_{w^-}^{\eps} \subset b_{w^+}^{\eps}$,
  \item the restriction of $g$ to $B_{w^-}^{\eps}$ is $\varepsilon$-Lipschitz, 
\end{enumerate}
then $g$ is $(2r,2\eps)$-proximal and $d_{\partial_{\infty}M}(w^+,x_g^+)\leq \eps$ and $d_{\partial_{\infty}M}(w^-,x_g^-) \leq \eps$.
\end{lema}

\subsection{A consequence of $(r,\eps)$-proximality}

Given our basepoint $o\in M$, recall the notation $|g|_M$ and $|g|_{M,\infty}$ from \eqref{eqn:length-stable-length}.
The following is analogous to Lemma~\ref{l.approachgap-proj}.

\begin{lema} \label{l.approachgap-hyp}
Let $(M,d_M)$ be a Gromov hyperbolic metric space and $o\in M$ a basepoint.
Then there exists $C'>0$ such that for any hyperbolic $g\in\Isom(M)$,
$$|g|_M - 2 (x_g^-|x_g^+)_o - C' \leq |g|_{M,\infty} \leq |g|_M. $$
\end{lema}

\begin{proof}
The inequality $|g|_{M,\infty} \leq |g|_M$ is immediate, and true for any $g\in\Isom(M)$.
Let us prove the other inequality.

For $x\in M\cup\partial_{\infty}M$, consider the Busemann function $\mathcal{B}_x : (M\cup\partial_{\infty}M) \times (M\cup\nolinebreak\partial_{\infty}M) \to \R$ defined by
\begin{equation} \label{eqn:Busemann}
\mathcal{B}_x(y,z) := (z|x)_y - (y|x)_z.
\end{equation}
Then there exists $C''>0$ with the following properties (see \cite[Cor.\,3.4.12 \& Prop.\,4.2.16]{dsu17}):
\begin{enumerate}[(i)]
  \item\label{item:Busemann-dist} $|d_M(y,z) - \mathcal{B}_x(y,z) - 2 (y|x)_z| \leq C''$
  for all $y,z\in M$ and $x\in M\cup\partial_{\infty}M$,
  \item\label{item:Busemann-dyn-deriv} for any hyperbolic element $g\in\Isom(M)$, there exists $c_g>0$ such that\linebreak $|\mathcal{B}_{x_g^-}(o,g^{-n}\cdot o) - n\,c_g| \leq C''$ for all $n\in\NN$.
\end{enumerate}
Applying \eqref{item:Busemann-dyn-deriv} to $g^n$ and \eqref{item:Busemann-dist} to $(x,y,z) = (x_g^-,o,g^{-n}\cdot o)$ gives that for any hyperbolic $g\in\Isom(M)$ and any $n\in\NN$,
$$n c_g \leq \mathcal{B}_{x_g^-}(o,g^{-n}\cdot o) + C'' \leq d_M(o,g^{-n}\cdot o) + 2 C''.$$
Dividing by~$n$ and taking the limit inferior, we obtain $c_g \leq |g|_{M,\infty}$.
Applying \eqref{item:Busemann-dyn-deriv} to~$g$, we obtain $\mathcal{B}_{x_g^-}(o,g^{-1}\cdot o) \leq |g|_{M,\infty} + C''$.

On the other hand, applying \eqref{item:Busemann-dist} to $(x,y,z) = (x_g^{\pm},o,g^{-1}\cdot o)$ gives

$$\mathcal{B}_{x_g^-}(o|g^{-1}\cdot o) + 2 (o|x_g^-)_{g^{-1}\cdot o} - C'' \leq d_M(o,g^{-1}\cdot o) \leq \mathcal{B}_{x_g^+}(o|g^{-1}\cdot o) + 2 (o|x_g^+)_{g^{-1}\cdot o} + C''.$$

By \eqref{item:Busemann-dyn-deriv} we have $\mathcal{B}_{x_g^-}(o|g^{-1}\cdot o) \geq -C''$ and $\mathcal{B}_{x_g^+}(o|g^{-1}\cdot o) = \mathcal{B}_{x_g^+}(g\cdot o|o) = - \mathcal{B}_{x_g^+}(o|g\cdot o) \leq C''$ (where the first equality follows from the $\Isom(M)$-invariance of the Gromov product).
Therefore $(o|x_g^-)_{g^{-1}\cdot o} \leq (o|x_g^+)_{g^{-1}\cdot o} + 2 C''$, and so
$$(o|x_g^-)_{g^{-1}\cdot o} \leq \min\big((o|x_g^-)_{g^{-1}\cdot o},(o|x_g^+)_{g^{-1}\cdot o}\big) + 2 C'' = \min\big((g\cdot o|x_g^-)_o,(g\cdot o|x_g^+)_o\big) + 2 C''.$$
The Gromov inequality \eqref{eqn:Gromov-ineq} then yields $(o|x_g^-)_{g^{-1}\cdot o} \leq (x_g^-|x_g^+)_o + 2 C'' + \delta$.
Applying again \eqref{item:Busemann-dist} to $(x,y,z) = (o,g^{-1}\cdot o,x_g^-)$, we finally obtain
\begin{align*}
|g|_M = d_M(o,g^{-1}\cdot o) & \leq \mathcal{B}_{x_g^-}(o|g^{-1}\cdot o) + 2 (o|x_g^-)_{g^{-1}\cdot o} + C''\\
& \leq \mathcal{B}_{x_g^-}(o|g^{-1}\cdot o) + 2 (x_g^-|x_g^+)_o + 5C'' + 2\delta\\
& \leq |g|_{M,\infty} + 2 (x_g^-|x_g^+)_o + C'
\end{align*}
where $C' := 6C'' + 2\delta > 0$.
\end{proof}

\begin{cor} \label{cor:r-eps-prox-length}
Let $(M,d_M)$ be a Gromov hyperbolic metric space and $o\in M$ a basepoint.
We endow the Gromov boundary $\partial_{\infty}M$ of~$M$ with a Bourdon metric $d_{\partial_{\infty}M} = d_{a,o}$ as in Section~\ref{subsubsec:Bourdon-metric}.
Then for any $r\geq\varepsilon>0$, there exists $C>0$ such that for any $g\in\Isom(M)$, if $g$ is $(r,\varepsilon)$-proximal in $\partial_{\infty}M$, then
$$|g|_M - C \leq |g|_{M,\infty} \leq |g|_M. $$
\end{cor}

\subsection{Proximal elements in semigroups with two global fixed points} 

The following lemma will be used in Sections \ref{sec:simult-AMS-simult-prox} and~\ref{sec:proof-exists-simult-prox} to treat the lineal case.

\begin{lema} \label{lem.linealcase}
Let $(M,d_M)$ be a Gromov hyperbolic metric space and $o\in M$ a basepoint.
We endow the Gromov boundary $\partial_{\infty}M$ with a Bourdon metric $d_{\partial_{\infty}M} = d_{a,o}$ as in Section~\ref{subsubsec:Bourdon-metric}.
Let $\Gamma$ be a subsemigroup of $\Isom(M)$ with two global fixed points $x_0,x_1$ in $\partial_{\infty}M$.
Then for any $r\geq\eps>0$ with $4r\leq d_{\partial_{\infty}M}(x_0,x_1)$, there exists $C>0$ such that any element $\gamma \in \Gamma$ satisfying $d_M(o,\gamma\cdot o) > C$ is $(r,\eps)$-proximal in $\partial_{\infty}M$.
\end{lema}

\begin{proof}
Fix $r\geq\eps>0$ with $4r\leq d_{\partial_{\infty}M}(x_0,x_1)$.
It is enough to check that for any sequence $(\gamma_n)\in\Gamma^{\NN}$ with $d_M(o,\gamma_n\cdot o) \to +\infty$, the elements $\gamma_n$ are $(r,\eps)$-proximal in $\partial_{\infty}M$ for all large enough~$n$.
Fix such a sequence $(\gamma_n)_{n\in\NN}$.

By Proposition~\ref{prop:Lip-hyp} with $\eps/2$ instead of~$\eps$, we can associate to each~$\gamma_n$ two points $y_{\gamma_n}^+,Y_{\gamma_n}^- \in \partial_{\infty}M$ such that for all large enough~$n$, we have $\gamma_n \cdot B_{Y_{\gamma_n}^-}^{\eps/2} \subset b_{y_{\gamma_n}^+}^{\eps/2}$ and the restriction of $\gamma_n$ to $B_{Y_{\gamma_n}^-}^{\eps/2}$ is $(\eps/2)$-Lipschitz.
Since $\eps\leq d_{\partial_{\infty}M}(x_0,x_1)$, for each $n\in\NN$ there exists $i\in\{0,1\}$ such that $d_{\partial_{\infty}M}(x_i,Y_{\gamma_n}^-) \geq \eps/2$, \ie $x_i \in B_{Y_{\gamma_n}^-}^{\eps/2}$.
We then have $x_i = \gamma_n\cdot x_i \in b_{y_{\gamma_n}^+}^{\eps/2}$.
Note that we cannot have $d_{\partial_{\infty}M}(x_{1-i},Y_{\gamma_n}^-) \geq \eps/2$, otherwise we would similarly have $x_{1-i} =\linebreak \gamma_n\cdot x_{1-i} \in b_{y_{\gamma_n}^+}^{\eps/2}$: contradiction since $d_{\partial_{\infty}M}(x_0,x_1) > \eps$.
Therefore $d_{\partial_{\infty}M}(x_{1-i},Y_{\gamma_n}^-) < \eps/2$, and $d_{\partial_{\infty}M}(y_{\gamma_n}^+,Y_{\gamma_n}^-) \geq d_{\partial_{\infty}M}(x_0,x_1) - \eps \geq 3r$.
Lemma~\ref{l.critprox-hyp} with $(r/2,\eps/2)$ instead of $(r,\eps)$ then implies that $\gamma_n$ is $(r,\eps)$-proximal in $\partial_{\infty}M$ for all large enough~$n$.
\end{proof}

\section{A condition allowing to move points at infinity} \label{s.movingpoints}

From now on we work in the following setting.

\begin{setting} \label{setting}
Let $\Gamma$ be a semigroup and let $I_1, I_2$ be finite sets.
For each $i\in I_1$, let $V_i$ be a Euclidean space of dimension $\geq 2$ and $\rho_i : \Gamma\to\GL(V_i)$ a representation such that $\rho_i(\Gamma)$ acts strongly irreducibly on~$V_i$.
For each $i\in I_2$, let $(M_i,d_{M_i})$ be a Gromov hyperbolic metric space, $o_i\in M_i$ a basepoint, and $\rho_i : \Gamma\to\Isom(M_i)$ a representation such that $\rho_i(\Gamma)$ acts on $\partial_\infty M_i$ without a unique global fixed point; we endow the Gromov boundary $\partial_\infty M_i$ with a Bourdon metric $d_i = d_{a_i,o_i}$ as in Section~\ref{subsubsec:Bourdon-metric}.
We denote by $I'_2$ the set of indices $i\in I_2$ such that $\rho_i(\Gamma)$ is of general type.
\end{setting}

Note that for each $i\in I_2\smallsetminus I'_2$ the semigroup $\rho_i(\Gamma)$ is lineal (see Section~\ref{subsubsec:isom-hyp}).

We consider the following condition, where for $i\in I_2$ we say that two points of $\partial_{\infty}M_i$ are transverse if they are distinct:
\begin{itemize}
  \item[(*)] for any finite sets $\cZ_i^+$, $i\in I_1\cup I'_2$, of points of $\PP(V_i)$ or $\partial_{\infty}M_i$, and any finite sets $\cZ_i^-$, $i\in I_1\cup I'_2$, of points of $\PP(V_i^*)$ or $\partial_{\infty}M_i$, there exists $\gamma \in \Gamma$ such that $\rho_i(\gamma)\cdot z_i^+$ is transverse to~$z_i^-$ and $z_i^+$ is transverse to $\rho_i(\gamma)\cdot z_i^-$ for all $z_i^+\in\cZ_i^+$ and all $z_i^-\in\cZ_i^-$, simultaneously for all $i\in I_1\cup I'_2$.
\end{itemize}

In this section we discuss some consequences of condition~(*) (namely Lemma~\ref{l.finiteset} and Corollary~\ref{cor:finiteset}), and then use them to prove that the condition actually always holds in the setting~\ref{setting} (Proposition~\ref{prop:move-away}).

\subsection{Creating elements in general position}

The following result should be compared with \cite[Lem.\,4.4]{ams95}.

\begin{lema} \label{l.finiteset}
In the setting~\ref{setting}, suppose that condition (*) holds.
For each $i\in I_1$, let $\cX_i^+$ (\resp $\cX_i^-$) be a finite subset of $\PP(V_i)$ (\resp $\PP(V_i^*)$).
For each $i\in I'_2$, let $\cX_i^+,\cX_i^-$ be finite subsets of $\partial_{\infty}M_i$.
Suppose that for any $i\in I_1\cup I'_2$,
\begin{itemize}
  \item the family $\cX_i^+$ is in general position,
  \item the family $\cX_i^-$ is in general position,
  \item any element of $\cX_i^+$ is transverse to any element of $\cX_i^-$.
\end{itemize}
Then there is an infinite family $\mathcal{F}_{\infty}$ of elements of~$\Gamma$ such that for any $i\in I_1\cup I'_2$ and any $\varepsilon\in\{1,-1\}$, 
\begin{itemize}
  \item the family $(\rho_i(\beta)^{\varepsilon}\cdot x_i^+)_{\beta\in\mathcal{F}_{\infty},\ x_i^+\in\cX_i^+}$ is in general position,
  \item the family $(\rho_i(\beta)^{\varepsilon}\cdot x_i^-)_{\beta\in\mathcal{F}_{\infty},\ x_i^-\in\cX_i^-}$ is in general position,
  \item $\rho_i(\beta)^{\varepsilon}\cdot x_i^+$ is transverse to $\rho_i(\beta')^{\varepsilon}\cdot x_i^-$ for all $\beta,\beta'\in\mathcal{F}_{\infty}$, all $x_i^+\in\cX_i^+$, and all $x_i^-\in\cX_i^-$.
\end{itemize}
\end{lema}

Here, for $i\in I_1$, we say that a family of points in $\PP(V_i)$ (\resp $\PP(V_i^*)$) is \emph{in general position} if any subset of $j$ points with $1\leq j \leq \dim(V_i)$ spans a $(j-1)$-dimensional projective subspace of $\PP(V_i)$ (\resp $\PP(V_i^*)$).
For $i\in I'_2$, we say that a family of points of $\partial_{\infty}M_i$ is \emph{in general position} if the points are pairwise distinct.

\begin{proof}[Proof of Lemma~\ref{l.finiteset}]
We construct the family $\mathcal{F}_{\infty} = \{ \beta_n \,|\, n\in\NN^*\}$ inductively.
We start by choosing any $\beta_1 \in \Gamma$. 
For $n\geq 2$, suppose that we have found $\beta_1, \ldots, \beta_{n-1} \in \Gamma$ such that for any $i\in I_1\cup I'_2$ and any $\varepsilon\in\{1,-1\}$, the family $(\rho_i(\beta_k)^{\varepsilon}\cdot x_i^+)_{1\leq k\leq n-1,\ x_i^+\in\cX_i^+}$ is in general position, the family $(\rho_i(\beta_k)^{\varepsilon}\cdot x_i^-)_{1\leq k\leq n-1,\ x_i^-\in\cX_i^-}$ is in general position, and $\rho_i(\beta_k)^{\varepsilon}\cdot x_i^+$ is transverse to $\rho_i(\beta_{\ell})^{\varepsilon}\cdot x_i^-$ for all $1\leq k,\ell\leq n-1$, all $x_i^+\in\cX_i^+$, and all $x_i^-\in\cX_i^-$.
Let us construct the next element~$\beta_n$.

For each $i\in I_1$, we choose a finite subset $\cW_i^-$of $\PP(V_i^*)$ (\resp $\cW_i^+$ of $\PP(V_i)$) corresponding to projective hyperplanes of $\PP(V_i)$ (\resp $\PP(V_i^*)$) containing all possible proper projective subspaces of $\PP(V_i)$ (\resp $\PP(V_i^*)$) generated by elements of $\bigcup_{k=1}^{n-1} \rho_i(\beta_k)\cdot\cX_i^+$ (\resp $\bigcup_{k=1}^{n-1} \rho_i(\beta_k)\cdot\cX_i^-$) and all possible proper projective subspaces of $\PP(V_i)$ (\resp $\PP(V_i^*)$) generated by elements of $\bigcup_{k=1}^{n-1} \rho_i(\beta_k)^{-1}\cdot\cX_i^+$ (\resp $\bigcup_{k=1}^{n-1} \rho_i(\beta_k)^{-1}\cdot\cX_i^-$).
For each $i\in I'_2$, we set
$$\cW_i^- := \bigcup_{k=1}^{n-1} \big(\rho_i(\beta_k)\cdot\cX_i^+ \cup \rho_i(\beta_k)^{-1}\cdot\cX_i^+\big) \quad\mathrm{and}\quad \cW_i^+ := \bigcup_{k=1}^{n-1} \big(\rho_i(\beta_k)\cdot\cX_i^- \cup \rho_i(\beta_k)^{\varepsilon}\cdot\cX_i^-\big).$$

By condition (*) with $\cZ_i^+ = \cX_i^+ \cup \bigcup_{k=1}^{n-1} \rho_i(\beta_k)\cdot\cX_i^+\cup \bigcup_{k=1}^{n-1} \rho_i(\beta_k)^{-1}\cdot\cX_i^+ \cup \cW_i^+$ and $\cZ_i^- = \cX_i^- \cup \bigcup_{k=1}^{n-1} \rho_i(\beta_k)\cdot\cX_i^- \cup \bigcup_{k=1}^{n-1} \rho_i(\beta_k)^{-1}\cdot\cX_i^- \cup \cW_i^-$, we can find an element $\beta_n \in \Gamma$ such that
\begin{itemize}
  \item for each $i\in I_1$ and each $\varepsilon\in\{1,-1\}$, the set $\rho_i(\beta_n)^{\varepsilon}\cdot\cX_i^+$ does not meet any proper projective subspace of $\PP(V_i)$ generated by elements of $\bigcup_{k=1}^{n-1} \rho_i(\beta_k)^{\varepsilon}\cdot\cX_i^+$, and the set $\rho_i(\beta_n)^{\varepsilon}\cdot\cX_i^-$ does not meet any proper projective subspace of $\PP(V_i^*)$ generated by elements of $\bigcup_{k=1}^{n-1} \rho_i(\beta_k)^{\varepsilon}\cdot\cX_i^-$,
  \item for each $i\in I'_2$ and each $\varepsilon\in\{1,-1\}$, we have $\rho_i(\beta_n)^{\varepsilon}\cdot\cX_i^+ \cap \bigcup_{k=1}^{n-1} \rho_i(\beta_k)^{\varepsilon}\cdot\cX_i^+ = \rho_i(\beta_n)^{\varepsilon}\cdot\cX_i^- \cap \bigcup_{k=1}^{n-1} \rho_i(\beta_k)^{\varepsilon}\cdot\nolinebreak\cX_i^- =\nolinebreak\emptyset$,
  \item $\rho_i(\beta_n)^{\varepsilon}\cdot x_i^+$ is transverse to $\rho_i(\beta_k)^{\varepsilon}\cdot x_i^-$ and $\rho_i(\beta_k)^{\varepsilon}\cdot x_i^+$ is transverse to $\rho_i(\beta_n)^{\varepsilon}\cdot x_i^-$ for all $\varepsilon\in\{1,-1\}$, all $1\leq k\leq n-1$, all $x_i^+\in\cX_i^+$, and all $x_i^-\in\cX_i^-$.
\end{itemize}
Then for any $i\in I_1\cup I'_2$ and any $\varepsilon\in\{1,-1\}$, the family $(\rho_i(\beta_k)^{\varepsilon}\cdot x_i^+)_{1\leq k\leq n,\ x_i^+\in\cX_i^+}$ is in general position, the family $(\rho_i(\beta_k)^{\varepsilon}\cdot x_i^-)_{1\leq k\leq n,\ x_i^-\in\cX_i^-}$ is in general position, and $\rho_i(\beta_k)^{\varepsilon}\cdot x_i^+$ is transverse to $\rho_i(\beta_{\ell})^{\varepsilon}\cdot x_i^-$ for all $1\leq k,\ell\leq n$, all $x_i^+\in\cX_i^+$, and all $x_i^-\in\cX_i^-$.
\end{proof}

\begin{cor} \label{cor:finiteset}
In the setting~\ref{setting}, suppose that condition (*) holds.
For each $i\in I_1$, let $\cX_i^+$ (\resp $\cX_i^-$) be a finite subset of $\PP(V_i)$ (\resp $\PP(V_i^*)$); we endow the projective space $\PP(V_i)$ with the distance function $d_i = d_{\PP(V_i)}$ from Section~\ref{sec:proj}.
For each $i\in I'_2$, let $\cX_i^+,\cX_i^-$ be finite subsets of $\partial_{\infty}M_i$.
Suppose that for any $i\in I_1\cup I'_2$, each of the two families $\cX_i^+$ and $\cX_i^-$ is in general position.
Then for any $N'\in\NN^*$, there exist $0<r_0<1$ and a finite subset $\mathcal{F}$ of~$\Gamma$ with the following property: for any set $\cY_i^-$ of $N'$ projective hyperplanes of $\PP(V_i)$, $i\in I_1$, and any set $\cY_i^-$ of $N'$ points of $\partial_\infty M_i$, $i\in I'_2$, and for any set $\cY_i^+$ of $N'$ points of $\PP(V_i)$, $i\in I_1$, and any set $\cY_i^+$ of $N'$ points of $\partial_\infty M_i$, $i\in I'_2$, there exists $\beta \in \mathcal{F}$ such that
$$d_i(\rho_i(\beta)\cdot x_i^+,y_i^-) \geq r_0 \quad\quad\mathrm{and}\quad\quad d_i(y_i^+,\rho_i(\beta)^{-1}\cdot x_i^-) \geq r_0$$
for all $(x_i^+,y_i^-) \in \cX_i^+\times\cY_i^-$ and all $(y_i^+,x_i^-) \in \cY_i^+\times\cX_i^-$, simultaneously for all $i\in I_1\cup I'_2$.
\end{cor}

Here we see each $x_i^- \in \cX_i^- \subset \PP(V_i^*)$ as a projective hyperplane of $\PP(V_i)$.

\begin{proof}
Fix $N'\in\NN^*$, and let us find $r_0\in (0,1)$ and $\mathcal{F}\subset\Gamma$ with the required property.
Let $\mathcal{F}_{\infty}$ be the infinite subset of~$\Gamma$ given by Lemma~\ref{l.finiteset}.
Let $\mathcal{F}$ be any finite subset of~$\mathcal{F}_{\infty}$ which is large enough in the sense that

\begin{equation} \label{eqn:card-F}
\# \mathcal{F} > 2N' \sum_{i\in I_1} (\dim(V_i)-1)\,(\#\cX_i^+ + \#\cX_i^-) + 2N' \sum_{i\in I'_2} (\#\cX_i^+ + \#\cX_i^-) .
\end{equation}
By construction, the families $(\rho_i(\beta)^{\varepsilon}\cdot x_i^+)_{\beta\in\mathcal{F},\ x_i^+\in\cX_i^+}$ and $(\rho_i(\beta)^{\varepsilon}\cdot x_i^-)_{\beta\in\mathcal{F},\ x_i^-\in\cX_i^-}$ are in general position for all $i\in I_1\cup I'_2$ and all $\varepsilon\in\{1,-1\}$.

Fix $i\in I_1$ and let $n_i := \dim(V_i)\geq 2$.
We claim that there exists $r_i>0$ such that for any projective hyperplane $y$ of $\PP(V_i)$ and any $x_i^+ \in \cX_i^+$, there are at most $n_i-1$ elements $\beta \in \mathcal{F}$ satisfying  $d_i(\rho_i(\beta)\cdot x_i^+,y)\leq r_i$. 
Indeed, if not, then for any $k\in\NN^*$ there exist a projective hyperplane $y_k$ of $\PP(V_i)$, an element $x_i^+ \in \cX_i^+$, and $n_i$ distinct elements $\beta_1,\dots,\beta_{n_i} \in \mathcal{F}$ such that $d_i(\rho_i(\beta_j)\cdot x_i^+,y_k)\leq 1/k$ for all $1\leq j\leq n_i$.
After passing to a subsequence and taking a limit, we obtain a projective hyperplane $y$ of $\PP(V_i)$, an element $x_i^+ \in \cX_i^+$, and $n_i$ distinct elements $\beta_1,\dots,\beta_{n_i} \in \mathcal{F}$ such that the points $\rho_i(\beta_j)\cdot x_i^+$ all belong to~$y$: impossible since the family $(\rho_i(\beta)\cdot x_i^+)_{\beta\in\mathcal{F},\ x_i^+\in\cX_i^+}$ is in general position.

Similarly, for each $i\in I_1$ there exists $r'_i>0$ such that for any point $y$ of $\PP(V_i)$ and any $x_i^- \in \cX_i^-$, there are at most $n_i-1$ elements $\beta \in \mathcal{F}$ such that $d_i(y,\rho_i(\beta)^{-1}\cdot x_i^-)\leq r'_i$.

Fix $i\in I'_2$.
Since the family $(\rho_i(\beta)\cdot x_i^+)_{\beta\in\mathcal{F},\ x_i^+\in\cX_i^+}$ is in general position, the minimal distance between two points in this family is positive.
Let $r_i>0$ be smaller than half this distance.
By the triangle inequality, for any $x_i^+ \in \cX_i^+$ and any $y\in\partial_{\infty}M_i$, there is at most one element $\beta \in \mathcal{F}$ such that $d_i(\rho_i(\beta)\cdot x_i^+,y)\leq r_i$.
Similarly, there exists $r'_i>0$ such that for any $y\in\partial_{\infty}M_i$ and any $x_i^- \in \cX_i^-$, there is at most one element $\beta \in \mathcal{F}$ such that $d_i(y,\rho_i(\beta)^{-1}\cdot x_i^-)\leq r'_i$.

Let us check that $\mathcal{F}$ and
$$r_0 := \min_{i\in I_1\cup I'_2} \min(r_i,r'_i) > 0$$
satisfy the required property.
Consider for each $i\in I_1$ a set $\cY_i^+$ (\resp $\cY_i^-$) of $N'$ points (\resp projective hyperplanes) of $\PP(V_i)$, and for each $i\in I'_2$ a set $\cY_i^+$ (\resp $\cY_i^-$) of $N'$ points of $\partial_\infty M_i$.
By the choice \eqref{eqn:card-F} of cardinality of $\mathcal{F}$ and the pigeonhole principle, we can find $\beta \in \mathcal{F}$ such that $d_i(\rho_i(\beta)\cdot x_i^+,y_i^-)\geq r_0$ for all $(x_i^+,y_i^-) \in \cX_i^+ \times \cY_i^-$ and $d_i(y_i^+,\rho_i(\beta)^{-1}\cdot x_i^-) \geq r_0$ for all $(y_i^+,x_i^-) \in \cY_i^+\times\cX_i^-$, simultaneously for all $i\in I_1\cup I'_2$.
\end{proof}

\subsection{Moving points away}

We now prove the following.

\begin{prop} \label{prop:move-away}
In the setting~\ref{setting}, condition (*) always holds.
\end{prop}

Our first step is to prove that condition (*) holds for $I_2'$ instead of $I_1 \cup I_2'$, and from this we then deduce the general case using Corollary~\ref{cor:finiteset} applied to~$I_2'$.
This first step is based on the following group-theoretic lemma (compare with \cite[Lem.\,4.1]{Neumann} in the case that $\Gamma$ is a group). 

\begin{lema} \label{lem:Neumann-semig}
Let $G$ be a group, $\Gamma$ a subsemigroup of~$G$, and $k\geq 1$ an integer.
For any subsets $H_1,\dots,H_k$ and $S_1,\dots,S_k$ of~$G$ with $S_i$ finite for all~$i$, if $\Gamma \subset \bigcup_{i=1}^k S_i H_i$, then there exist $1\leq i\leq k$ and a finite subset $F_i$ of~$G$ such that $\Gamma \subset F_i H_i H_i^{-1} F_i^{-1}$.
\end{lema}

\begin{proof}[Proof of Lemma~\ref{lem:Neumann-semig}]
We argue by induction on~$k$.

Suppose $k=1$.
Let $H_1,S_1$ be subsets of~$G$, with $S_1$ finite, such that $\Gamma \subset S_1 H_1$.
Since $\Gamma$ is a semigroup, for any $\gamma\in\Gamma$ we have $\gamma S_1 H_1 \cap S_1 H_1 \neq \emptyset$, hence $\gamma \in S_1 H_1 H_1^{-1} S_1^{-1}$.

Suppose $k\geq 2$.
Assuming the result is true for $k-1$, let us prove it for~$k$.
Let $H_1,\dots,H_k$ and $S_1,\dots,S_k$ be subsets of~$G$, with $S_i$ finite for all~$i$, such that $\Gamma \subset \bigcup_{i=1}^k S_i H_i$.
If $\Gamma \subset S_k H_k H_k^{-1} S_k^{-1}$, then we are done.
So assume that this is not the case.
Then there exists $\gamma \in \Gamma$ that does not belong to $S_k H_k H_k^{-1} S_k^{-1}$, or in other words such that $\gamma S_k H_k \cap S_k H_k = \emptyset$.
Since $\Gamma$ is a semigroup and $\Gamma \subset \bigcup_{i=1}^k S_i H_i$, we must have $\gamma (S_k H_k \cap \Gamma) \subset \bigcup_{i=1}^{k-1} S_i H_i$.
Then $\Gamma = \bigcup_{i=1}^k (S_i H_i\cap\Gamma) \subset \bigcup_{i=1}^{k-1} (S_i \cup \gamma^{-1}S_i) H_i$.
By induction, there exists $1\leq i\leq k-1$ and a finite subset $F_i$ of~$G$ such that $\Gamma \subset F_i H_i H_i^{-1} F_i^{-1}$.
\end{proof}

\begin{proof}[Proof of Proposition~\ref{prop:move-away}]
We first observe that for any $i\in I'_2$, condition (*) holds for $\{ i\}$ instead of $I_1\cup I'_2$.
Indeed, since the semigroup $\rho_i(\Gamma)$ is of general type, for any finite subsets $\cZ_i^+,\cZ_i^-$ of $\partial_{\infty}M_i$, there exists $\gamma\in\Gamma$ such that $\rho_i(\gamma)$ is hyperbolic with attracting and repelling fixed points both outside $\cZ_i^+\cup\cZ_i^-$ (see \cite[Prop.\,7.3.1--7.4.7]{dsu17}); applying a large power of $\rho_i(\gamma)$ to the set $\cZ_i^+\cup\cZ_i^-$ then takes it away from itself.

We next observe that condition (*) holds for $I'_2$ instead of $I_1\cup I'_2$.
Indeed, let $\sigma := (\rho_i)_{i\in I'_2} : \Gamma \to \prod_{i\in I'_2} \Isom(M_i)$.
Suppose by contradiction that condition (*) fails for~$I'_2$: there exist finite subsets $\cZ_i^+,\cZ_i^-$ of $\partial_{\infty}M_i$, for $i\in I'_2$, such that for any $\gamma\in\Gamma$ there exists $i\in I'_2$ with $\rho_i(\gamma)\cdot\cZ_i^+ \cap \cZ_i^- \neq \emptyset$ or $\cZ_i^+ \cap \rho_i(\gamma)\cdot\cZ_i^- \neq \emptyset$.
Then $\sigma(\Gamma) \subset \bigcup_{i\in I'_2} \bigcup_{z\in\cZ_i^+\cup\cZ_i^-} S_z H_z$ where $H_z = \mathrm{stab}_{\Isom(M_i)}(z) \times \prod_{i'\in I'_2\smallsetminus\{i\}} \Isom(M_{i'})$ and $S_z$ is some finite subset of $\prod_{i'\in I'_2} \Isom(M_{i'})$.
Applying Lemma~\ref{lem:Neumann-semig} to the group $\prod_{i'\in I'_2} \Isom(M_{i'})$ and its subsemigroup $\sigma(\Gamma)$, we obtain the existence of $i\in I'_2$, of a point $z\in\cZ_i^+\cup\cZ_i^-$, and of a finite subset $F$ of $\prod_{i'\in I'_2} \Isom(M_{i'})$ such that $\sigma(\Gamma) \subset F H_z F^{-1}$.
(Note that $H_z H_z^{-1} = H_z$ since $H_z$ is a group.)
Taking the projection to $\Isom(M_i)$, we obtain a finite subset $F_i$ of $\Isom(M_i)$ such that $\rho_i(\gamma)\cdot (F_i\cdot z) \cap (F_i\cdot z) \neq \emptyset$ for all $\gamma\in\Gamma$, and so condition~(*) fails for~$\{i\}$: contradiction.
Thus condition (*) holds for $I'_2$ instead of $I_1\cup I'_2$.

Let us now check that condition (*) holds for $I_1\cup I'_2$.
For this, choose arbitrary finite sets $\cZ_i^+$, $i\in I_1\cup I'_2$, of points of $\PP(V_i)$ or $\partial_{\infty}M_i$, and arbitrary finite sets $\cZ_i^-$, $i\in I_1\cup I'_2$, of points of $\PP(V_i^*)$ or $\partial_{\infty}M_i$.
We wish to find $\gamma \in \Gamma$ such that $\rho_i(\gamma)\cdot z_i^+$ is transverse to~$z_i^-$ and $z_i^+$ is transverse to $\rho_i(\gamma)\cdot z_i^-$ for all $z_i^+\in\cZ_i^+$ and all $z_i^-\in\cZ_i^-$, simultaneously for all $i\in I_1\cup I'_2$.

Let $\rho := (\rho_i)_{i\in I_1} : \Gamma \to \prod_{i\in I_1} \GL(V_i)$, let $G$ be the Zariski closure of $\rho(\Gamma)$ in $\prod_{i\in I_1} \GL(V_i)$, and for any $i\in I_1$ let $\pi_i : \prod_{i'\in I_1} \GL(V_{i'})\to\GL(V_i)$ be the natural projection.
For any $i\in I_1$, let $\mathcal{U}_i$ be the set of elements $g\in G$ such that $\pi_i(g)\cdot z_i^+$ is transverse to~$z_i^-$ and $z_i^+$ is transverse to $\pi_i(g)\cdot z_i^-$ for all $z_i^+\in\cZ_i^+$ and all $z_i^-\in\cZ_i^-$.
Since $\cZ_i^+$ is a finite set of points of $\PP(V_i)$, since $\cZ_i^-$ can be seen as a finite set of proper algebraic subvarieties of $\PP(V_i)$, and since the action of $\rho_i(\Gamma)$ on~$V_i$ is strongly irreducible, $\mathcal{U}_i$ is a nonempty Zariski open subset of~$G$ (see \eg \cite[\S\,6.1]{BenoistQuint-livre}).
Therefore $\mathcal{U} := \bigcap_{i\in I_1} \mathcal{U}_i$ is also a nonempty Zariski open subset of~$G$.

Let $\Gamma'$ be the set of elements $\gamma\in\Gamma$ such that $\rho_i(\gamma)\cdot (\cZ_i^+\cup\cZ_i^-) \cap (\cZ_i^+\cup\cZ_i^-) = \emptyset$ for all $i\in I'_2$.
By Corollary~\ref{cor:finiteset} for $I'_2$ instead of $I_1\cup I'_2$ (with $\cX_i^- = \cZ_i^+ \cup\cZ_i^-$ and $N' = \# (\cZ_i^+\cup\cZ_i^-)$, taking $\cY_i^+ = \rho_i(\gamma)\cdot (\cZ_i^+\cup\cZ_i^-)$), there is a finite subset $\mathcal{F}$ of~$\Gamma$ such that for any $\gamma\in\Gamma$, at least one element $\beta\gamma$, for $\beta\in\mathcal{F}$, belongs to~$\Gamma'$.

Since $\mathcal{U}$ is a nonempty Zariski open subset of~$G$, so is $\bigcap_{\beta\in\mathcal{F}}\, \rho(\beta)^{-1}\,\mathcal{U}$.
Since $\rho(\Gamma)$ is Zariski dense in~$G$, there exists $\gamma \in \Gamma$ such that $\rho(\gamma) \in \bigcap_{\beta\in\mathcal{F}}\, \rho(\beta)^{-1}\,\mathcal{U}$.
By the above, there exists $\beta\in\mathcal{F}$ such that $\gamma' := \beta \gamma$ belongs to~$\Gamma'$.
We have $\rho(\gamma') \in \mathcal{U}$, and so $\rho_i(\gamma)\cdot z_i^+$ is transverse to~$z_i^-$ and $z_i^+$ is transverse to $\rho_i(\gamma)\cdot z_i^-$ for all $z_i^+\in\cZ_i^+$ and all $z_i^-\in\cZ_i^-$, simultaneously for all $i\in I_1\cup I'_2$, as in condition~(*).
\end{proof}

\subsection{Simultaneously moving the basepoint}

The next lemma will be used in the proof of Proposition~\ref{prop:simult-prox-proper}.
For a semigroup~$\Gamma$, we define an \emph{infinite-index} subset to be a subset $H \subset \Gamma$ such that there is no finite subset $S$ of~$\Gamma$ with $\Gamma \subset S H$.

\begin{lema} \label{lem:simult-basept-infty}
Let $\Gamma$ be a semigroup and $I$ a finite set.
For each $i\in I$, let $(M_i,d_{M_i})$ be a Gromov hyperbolic metric space, $o_i\in M_i$ a basepoint, and $\rho_i : \Gamma \to \mathrm{Isom}(M_i)$ a representation such that $\{ d_{M_i}(o_i,\rho_i(\gamma)\cdot o_i) \,|\, \gamma\in\Gamma\}$ is unbounded.
Then for every $C>0$ there exists an element $\gamma\in\Gamma$ such that $d_{M_i}(o_i,\rho_i(\gamma)\cdot\nolinebreak o_i)\geq C$ for all $i\in I$; more precisely, there exists an infinite-index subset $H$ of~$\Gamma$ such that $d_{M_i}(o_i,\rho_i(\gamma)\cdot\nolinebreak o_i)\geq C$ for all $\gamma\in\Gamma\smallsetminus H$ and all $i\in I$.
\end{lema} 

\begin{proof}
Fix $C>0$.
For each $i\in I$, let $H_i$ be the subset of~$\Gamma$ consisting of those elements $\gamma\in\Gamma$ such that $d_{M_i}(o_i,\rho_i(\gamma)\cdot o_i)<C$.
Let us show that $H := \bigcup_{i\in I} H_i$ is infinite-index in~$\Gamma$.

For this we start by observing that for any $i\in I$ and any $g,g'\in\Isom(M_i)$, we have $d_{M_i}(o_i,g\cdot o_i) = d_{M_i}(o_i,g^{-1}\cdot o_i)$ and, by the triangle inequality, 
\begin{equation} \label{eqn:triangle-ineq-dist-o}
d_{M_i}(o_i,gg'\cdot o_i) \leq d_{M_i}(o_i,g\cdot o_i) + d_{M_i}(g\cdot o_i,gg'\cdot o_i) = d_{M_i}(o_i,g\cdot o_i) + d_{M_i}(o_i,g'\cdot o_i).
\end{equation}

Suppose by contradiction that there is a finite subset $S$ of~$\Gamma$ such that $\Gamma \subset SH =  \bigcup_{i\in I} S H_i$.
Let $\sigma := (\rho_i)_{i\in I} : \Gamma \to \prod_{i\in I} \Isom(M_i)$.
Applying Lemma~\ref{lem:Neumann-semig} to the group $\prod_{i\in I} \Isom(M_i)$ and its subsemigroup $\sigma(\Gamma)$, we obtain the existence of $i\in I$ and of a finite subset $F$ of $\prod_{i'\in I} \Isom(M_{i'})$ such that $\sigma(\Gamma) \subset F \sigma(H_i) \sigma(H_i)^{-1} F^{-1}$.

Taking the projection to $\Isom(M_i)$, we obtain a finite subset $F_i$ of $\Isom(M_i)$ such that for any $\gamma\in\Gamma$, there exist $\sigma,\sigma'\in F_i$ and $g,g'\in\rho_i(H_i)$ with $\rho_i(\gamma) = \sigma g {g'}^{-1} {\sigma'}^{-1}$; by \eqref{eqn:triangle-ineq-dist-o}, we have
\begin{eqnarray*}
d_{M_i}(o_i,\rho_i(\gamma)\cdot o_i) & \leq & d_{M_i}(o_i,\sigma\cdot o_i) + d_{M_i}(o_i,g\cdot o_i) + d_{M_i}(o_i,g'\cdot o_i) + d_{M_i}(o_i,\sigma'\cdot o_i)\\
& \leq & 2C + 2\,\max_{\sigma''\in F_i} d_{M_i}(o_i,\sigma''\cdot o_i),
\end{eqnarray*}
hence the set $\{ d_{M_i}(o_i,\rho_i(\gamma)\cdot o_i) \,|\, \gamma\in\Gamma \}$ is bounded: contradiction.
\end{proof} 

\begin{remark}
We cannot argue similarly to show that if for each $i\in I$, we have a Euclidean space $V_i$ of dimension $\geq 2$ and a representation $\rho_i : \Gamma\to\GL(V_i)$ such that $\{ (\mu_1-\mu_2)(\rho_i(\gamma)) \,|\, \gamma\in\Gamma\}$ is unbounded, then for every $C>0$ there exists $\gamma\in\Gamma$ with $(\mu_1-\mu_2)(\rho_i(\gamma))\geq C$ for all $i\in I$.
Indeed, it is \emph{not} true in general that $(\mu_1-\mu_2)(gg') \leq (\mu_1-\mu_2)(g) + (\mu_1-\mu_2)(g')$.
\end{remark}

\section{A simultaneous Abels--Margulis--Soifer lemma assuming the existence of a simultaneously proximal element} \label{sec:simult-AMS-simult-prox}

In this section we prove the following version of Theorem~\ref{thm:main}, assuming the existence of a simultaneously proximal element.
This existence will be shown to hold in Section~\ref{sec:proof-exists-simult-prox}.

\begin{teo} \label{teo.generalresult}
In the setting~\ref{setting}, suppose that there exists $\gamma_0\in\Gamma$ such that $\rho_i(\gamma_0)$ is proximal in $\PP(V_i)$ for all $i\in I_1$ and $\rho_i(\gamma_0)$ is proximal in $\partial_{\infty}M_i$ for all $i\in I_2$.
Then the conclusion of Theorem~\ref{thm:main} holds, \ie there exists $r_0>0$ such that for any $r_0\geq r\geq\eps>0$, there is a finite subset $S$ of~$\Gamma$ with the following property: for any $\gamma \in \Gamma$, we can find $s\in S$ such that $\rho_i(\gamma s)$ is $(r,\eps)$-proximal in $\PP(V_i)$ for all $i\in I_1$ and $\rho_i(\gamma s)$ is $(r,\eps)$-proximal in $\partial_{\infty}M_i$ for all $i \in I_2$.
\end{teo}

\subsection{A preliminary reduction} \label{subsec:reduction}

Recall the notation $\rho_i(\Gamma)_{\mathrm{fix}}$ for lineal semigroups from Section~\ref{subsubsec:isom-hyp}.

\begin{lema} \label{lem:lineal-fix}
If $\Gamma$ and $(\rho_i)_{i\in I_1\cup I_2}$ are as in the setting~\ref{setting}, then so are $\Gamma_0$ and $(\rho_i|_{\Gamma_0})_{i\in I_1\cup I_2}$, where $\Gamma_0$ is the subsemigroup of~$\Gamma$ consisting of those elements $\gamma\in\Gamma$ such that $\rho_i(\gamma) \in \rho_i(\Gamma)_{\mathrm{fix}}$ for all $i\in I_2\smallsetminus I'_2$.
Moreover, there is a finite subset $S_0$ of $\Gamma$ such that for each $\gamma\in\Gamma$, we can find $s_0\in S_0$ with $\gamma s_0 \in \Gamma_0$.
\end{lema}

\begin{proof}
We first observe that there is a finite subset $S_0$ of $\Gamma$ such that for each $\gamma\in\Gamma$, we can find $s_0\in S_0$ with $\gamma s_0 \in \Gamma_0$.
Indeed, for each $i \in I_2\smallsetminus I'_2$, we can define a representation $\varphi_i : \Gamma \to \ZZ/2\ZZ$ of the semigroup~$\Gamma$ by $\varphi_i(\gamma) = 0$ if $\rho_i(\gamma)  \in \rho_i(\Gamma)_{\mathrm{fix}}$, and $\varphi_i(\gamma) = 1$ otherwise.
Then $\varphi = (\varphi_i)_{i \in I_2\smallsetminus I'_2}: \Gamma \to (\ZZ/2\ZZ)^{I_2\smallsetminus I'_2}$ is a representation, and $\Gamma_0 = \varphi^{-1}(0)$.
Let $S_0$ be a finite subset of~$\Gamma$ containing an element of $\varphi^{-1}(x)$ for each $x \in \varphi(\Gamma)$.
For any $\gamma \in \Gamma$, there exists $s_0\in S_0$ such that $\varphi(\gamma) = \varphi(s_0)$; then $\varphi(\gamma s_0) = \varphi(s_0)^2 = 0$, and so $\gamma s_0\in\Gamma_0$.

Let $i\in I_1$.
By assumption there exists $\gamma\in\Gamma$ such that $\rho_i(\gamma)$ is proximal in $\PP(V_i)$; then $\gamma^2\in\Gamma_0$ and $\rho_i(\gamma^2)$ is still proximal in $\PP(V_i)$.
We claim that $\rho_i(\Gamma_0)$ acts strongly irreducibly on~$V_i$.
Indeed, consider a $\rho_i(\Gamma_0)$-invariant finite union of linear subspaces of~$V_i$.
By the observation above, the image of this finite union under $\rho_i(\Gamma)$ is still a finite union of linear subspaces of~$V$, now $\rho_i(\Gamma)$-invariant.
Since $\rho_i(\Gamma)$ acts strongly irreducibly on~$V_i$, this finite union must be $\{0\}$ or~$V_i$, which shows that $\rho_i(\Gamma_0)$ acts strongly irreducibly on~$V_i$.

Let $i\in I_2$.
For any element $\gamma\in\Gamma$ such that $\rho_i(\gamma)$ is proximal in $\partial_{\infty}M_i$, the element $\gamma^2$ belongs to~$\Gamma_0$ and $\rho_i(\gamma^2)$ is still proximal in $\partial_{\infty}M_i$, with the same fixed points.
In particular, if $\rho_i(\Gamma)$ is of general type (\resp lineal), then so is  $\rho_i(\Gamma_0)$.
\end{proof}

\subsection{Proof of Theorem~\ref{teo.generalresult}}

By Lemma~\ref{lem:lineal-fix}, we may and shall assume that $\rho_i(\gamma) \in \rho_i(\Gamma)_{\mathrm{fix}}$ for all $i\in I_2\smallsetminus I'_2$.

For any $i\in I_1$, let $x_i^+ \in \PP(V_i)$ be the attracting fixed point of $\rho_i(\gamma_0)$, and $X_i^- \subset \PP(V_i)$ its repelling hyperplane.
For any $i\in I_2$, let $x_i^+ \in \partial_\infty M_i$ be the attracting fixed point of $\rho_i(\gamma_0)$, and $X_i^- \in \partial_\infty M_i$ its repelling fixed point.
By Proposition~\ref{prop:move-away}, condition (*) holds.
For $\cX_i^+ = \{ x_i^+\}$ and $\cX_i^- = \{ X_i^-\}$, with $N'=1$, let $r_0>0$ be the constant and $\mathcal{F}$ the finite subset of~$\Gamma$ given by Corollary~\ref{cor:finiteset}.
Let $r_1 := \min_{i\in I_1\cup I_2} d_i(x_i^+,X_i^-) > 0$, where for $i\in I_1$ we denote by $d_i = d_{\PP(V_i)}$ the distance function on $\PP(V_i)$ from Section~\ref{sec:proj}.

Note that each $\rho_i(\beta)$ and $\rho_i(\beta)^{-1}$ for $i\in I_1\cup I_2$ and $\beta\in\Gamma$ is globally Lipschitz.
Indeed, $\rho_i(\beta)$ is a diffeomorphism of the compact space $\mathbb{P}(V_i)$ for $i\in I_1$, and for $i\in I_2$ one can, for instance, apply Lemma~\ref{lem:Shad}.\eqref{item:Shad-2} and Proposition~\ref{prop:Lip-hyp-Shad}.\eqref{item:Lip-hyp-Shad-1}.
Since $I_1\cup I_2$ and $\mathcal{F}$ are finite sets, there exists $D_0\geq 1$ such that $\rho_i(\beta)$ and $\rho_i(\beta)^{-1}$ are $D_0$-Lipschitz for all $i\in I_1\cup I_2$ and all $\beta\in\mathcal{F}$.

Fix any $0 < \eps \leq r \leq \min(r_0/3,r_1/4)$, and let us show the existence of a finite subset $S$ of~$\Gamma$ with the property that for any $\gamma \in \Gamma$, we can find $s\in S$ such that $\rho_i(\gamma s)$ is $(r,\eps)$-proximal in $\PP(V_i)$ or $\partial_{\infty}M_i$ for all $i \in I_1\cup I_2$.

For each $i$ in $I_1$ (\resp $I_2$) and each $\gamma\in\Gamma$, let $Y_{\rho_i(\gamma)}^-$ be given by Proposition~\ref{prop:Lip-proj} (\resp Proposition~\ref{prop:Lip-hyp} for $\eps/2$ instead of~$\eps$).
Then there exists $D\geq D_0\geq 1$ such that
\begin{enumerate}[(a)]
  \item\label{item:a} for any $i\in I_1\cup I'_2$ and any $\gamma\in\Gamma$, the restriction of $\rho_i(\gamma)$ to $B^{\eps/2}_{Y_{\rho_i(\gamma)}^-}$ is $D$-Lipschitz.
\end{enumerate}
Choose $\eps'>0$ small enough so that $2D^3\eps' \leq \eps$.
There exists $n_0\in\NN$ such that for any $n\geq n_0$,
\begin{enumerate}[(a)] \setcounter{enumi}{1}
  \item\label{item:b} $\rho_i(\gamma_0^n)$ is $(r,\eps')$-proximal for all $i\in I_1\cup I'_2$.
\end{enumerate}
Choose such an~$n$ and let us check that the finite set
\begin{equation} \label{eqn:S_n}
S_n := \{ \beta \gamma_0^n \beta' ~|~ \beta,\beta'\in\mathcal{F} \}
\end{equation}
satisfies that for any $\gamma\in\Gamma$ we can find $s\in S_n$ such that $\rho_i(\gamma s)$ is $(r,\eps)$-proximal in $\PP(V_i)$ or $\partial_{\infty}M_i$ for all $i \in I_1\cup I'_2$.

Consider an element $\gamma \in \Gamma$.
By Corollary~\ref{cor:finiteset} with $\cY_i^- = \{ Y_{\rho_i(\gamma)}^-\}$, we can find $\beta\in\mathcal{F}$ such that
\begin{enumerate}[(a)] \setcounter{enumi}{2}
  \item\label{item:c} $d_i(\rho_i(\beta)\cdot x_i^+, Y_{\rho_i(\gamma)}^-) \geq 3r$ for all $i\in I_1\cup I_2'$, 
\end{enumerate}
and by Corollary~\ref{cor:finiteset} with $\cY_i^+ = \{ \rho_i(\gamma\beta)\cdot x_i^+\}$, we can find $\beta'\in\mathcal{F}$ such that
\begin{enumerate}[(a)] \setcounter{enumi}{3}
  \item\label{item:d} $d_i(\rho_i(\gamma\beta)\cdot x_i^+, \rho_i(\beta')^{-1}\cdot X_i^-) \geq 3r$ for all $i \in I_1\cup I_2'$. 
\end{enumerate}

We claim that for any $i\in I_1\cup I'_2$ the element $\rho_i(\gamma \beta \gamma_0^n \beta')$ sends $B^{\eps/2}_{\rho_i(\beta')^{-1}\cdot X_i^-}$ into $b_{\rho_i(\gamma\beta)\cdot x_i^+}^{\eps/2}$ in an $(\eps/2)$-Lipschitz way.
Indeed, this is a direct consequence of the following four observations:
\begin{enumerate}
  \item $\rho_i(\beta)$ and $\rho_i(\beta)^{-1}$ are $D$-Lipschitz and $2D\eps' \leq \eps$, hence $\rho_i(\beta)$ sends $B^{\eps/2}_{\rho_i(\beta')^{-1}\cdot X_i^-}$ into $B^{\eps'}_{X_i^-}$;
  \item by \eqref{item:b}, the element $\rho_i(\gamma_0^n)$ sends $B_{X_i^-}^{\eps'}$ into $b_{x_i^+}^{\eps'}$ in an $\eps'$-Lipschitz way;
  \item $\rho_i(\beta)$ is $D$-Lipschitz, hence sends $b_{x_i^+}^{\eps'}$ into $b_{\rho_i(\beta)\cdot x_i^+}^{D\eps'}$;
  \item by \eqref{item:c}, since $D\eps' + \eps/2 \leq 3r$, we have $b_{\rho_i(\beta)\cdot x_i^+}^{D\eps'} \subset B^{\eps/2}_{Y_{\rho_i(\gamma)}^-}$, and so by \eqref{item:a} the element $\rho_i(\gamma)$ is $D$-Lipschitz on $b_{\rho_i(\beta)\cdot x_i^+}^{D\eps'}$ and sends it into $b_{\rho_i(\gamma\beta)\cdot x_i^+}^{D^2\eps'} \subset b_{\rho_i(\gamma\beta)\cdot x_i^+}^{\eps/2}$.
\end{enumerate}
Finally, by \eqref{item:d} we have $d(\rho_i(\gamma\beta)\cdot x_i^+, \rho_i(\beta')^{-1}\cdot X_i^-) \geq 3r$.
Lemmas \ref{l.critprox-proj} and~\ref{l.critprox-hyp} with $(r/2,\eps/2)$ instead of $(r,\eps)$ then imply that $\rho_i(\gamma \beta \gamma_0^n \beta')$ is $(r,\eps)$-proximal in $\PP(V_i)$ or $\partial_{\infty}M_i$ for all $i\in I_1\cup I_2'$.

This shows that for any $n\geq n_0$, the set $S_n$ of \eqref{eqn:S_n} has the property that for any $\gamma\in\Gamma$ we can find $s\in S_n$ such that $\rho_i(\gamma s)$ is $(r,\eps)$-proximal in $\PP(V_i)$ or $\partial_{\infty}M_i$ for all $i \in I_1 \cup I'_2$.

If $I_2 = I'_2$, then we are done, so let us now assume $I_2\smallsetminus I_2' \neq \emptyset$.
By Lemma~\ref{lem.linealcase}, there exists $C>0$ such that for any $i\in I_2\smallsetminus I_2'$ and any $\gamma\in\Gamma$, if $d_{M_i}(o_i,\rho_i(\gamma)\cdot o_i) \geq C$, then $\rho_i(\gamma)$ is $(r,\eps)$-proximal in $\partial_{\infty}M_i$.
Let $N := \#(I_2\smallsetminus I_2')$.
Since for every $i\in I_2\smallsetminus I_2'$ and every $\beta\in\mathcal{F}$ we have $d_{M_i}(o_i,\rho_i(\beta\gamma_0^n\beta')\cdot o_i) \to +\infty$ as $n\to +\infty$ (see \eqref{eqn:triangle-ineq-dist-o}), we can find integers $n_0 \leq n_1 < \dots < n_{N+1}$ such that for any $i\in I_2\smallsetminus I_2'$, the intervals
$$\mathcal{I}_{i,j} := \Big[ \min_{\beta,\beta'\in\mathcal{F}} d_{M_i}(o_i,\rho_i(\beta\gamma_0^{n_j}\beta')\cdot o_i) - C, \max_{\beta,\beta'\in\mathcal{F}} d_{M_i}(o_i,\rho_i(\beta\gamma_0^{n_j}\beta')\cdot o_i) + C \Big],$$
for $1\leq j\leq N+1$, are pairwise disjoint.
We claim that
$$S := S_{n_1} \cup \dots \cup S_{n_{N+1}}$$
satisfies the property of Theorem~\ref{teo.generalresult}, where the $S_{n_j}$ are given by \eqref{eqn:S_n}.
Indeed, consider $\gamma\in\Gamma$.
Since $\#(I_2\smallsetminus I_2') = N$, by the pigeonhole principle there must exist $1\leq j\leq N+1$ such that $d_{M_i}(o_i,\rho_i(\gamma)\cdot o_i) \notin \mathcal{I}_{i,j}$ for all $i\in I_2\smallsetminus I_2'$.
By the triangle inequality, we then have $d_{M_i}(o_i,\rho_i(\gamma s)\cdot o_i)\geq C$ for all $i\in I_2\smallsetminus I_2'$ and all $s\in S_{n_j}$, hence $\rho_i(\gamma s)$ is $(r,\eps)$-proximal in $\partial_{\infty}M_i$ for all $i\in I_2\smallsetminus I_2'$ and all $s\in S_{n_j}$.
By the above, there exists $s\in S_{n_j}$ such that $\rho_i(\gamma s)$ is $(r,\eps)$-proximal in $\PP(V_i)$ or $\partial_{\infty}M_i$ for all $i\in I_1\cup I'_2$, hence for all $i\in I_1\cup I_2$.

\section{Existence of a simultaneously proximal element} \label{sec:proof-exists-simult-prox}

In this section we complete the proof of Theorem~\ref{thm:main} by establishing the following.

\begin{prop} \label{prop:simult-prox-proper}
In the setting~\ref{setting}, there exists $\gamma\in\Gamma$ such that $\rho_i(\gamma)$ is proximal in $\PP(V_i)$ or $\partial_{\infty}M_i$ simultaneously for all $i\in I_1\cup I_2$.
\end{prop}

\begin{proof}
Let $\Gamma_0$ be the subsemigroup of~$\Gamma$ consisting of those elements $\gamma\in\Gamma$ such that $\rho_i(\gamma) \in \rho_i(\Gamma)_{\mathrm{fix}}$ for all $i\in I_2\smallsetminus I'_2$.
By Lemma~\ref{lem:lineal-fix}, if we replace $\Gamma$ by~$\Gamma_0$, then we are still in the setting~\ref{setting}.
Therefore we may and shall assume that $\Gamma = \Gamma_0$.

$\bullet$ \textbf{Step 1: $I_1\cup I'_2$.}
We first assume $I_1\cup I'_2 \neq \emptyset$, and prove that there exists $\gamma\in\Gamma$ such that $\rho_i(\gamma)$ is proximal for all $i\in I_1\cup I'_2$.
For this, by induction, it is sufficient to prove that for any nonempty subset $I$ of $I_1\cup I'_2$ and any $i_0 \in (I_1\cup I'_2)\smallsetminus I$, if there exists $\gamma_1\in\Gamma$ such that $\rho_i(\gamma_1)$ is proximal for all $i\in I$, then there exists $\gamma\in\Gamma$ such that $\rho_i(\gamma)$ is proximal for all $i\in I\cup\{i_0\}$.
Let us prove this.

By assumption, there exists $\gamma_0\in\Gamma$ such that $\rho_{i_0}(\gamma_0)$ is proximal, with attractor $x_{i_0}^+$ and repellor $X_{i_0}^-$.
For each $i\in I$, let $x_i^+$ (\resp $X_i^-$) be the attractor (\resp repellor) of $\rho_i(\gamma_1)$.
By Proposition~\ref{prop:move-away}, condition (*) holds.
For $\cX_i^+ = \{ x_i^+\}$ and $\cX_i^- = \{ X_i^-\}$, with $N'=1$, let $r_0>0$ be the constant and $\mathcal{F}$ the finite subset of~$\Gamma$ given by Corollary~\ref{cor:finiteset}.
Let $r_1 := \min_{i\in I_1\cup I'_2} d_i(x_i^+,X_i^-) > 0$, where for $i\in I_1$ we denote by $d_i = d_{\PP(V_i)}$ the distance function on $\PP(V_i)$ from Section~\ref{sec:proj}.
As in the proof of Theorem~\ref{teo.generalresult}, there exists $D_0\geq 1$ such that $\rho_i(\beta)$ and $\rho_i(\beta)^{-1}$ are $D_0$-Lipschitz for all $i\in I_1\cup I_2$ and all $\beta\in\mathcal{F}$.

Fix $0 < \eps \leq r \leq \min(r_0,r_1)/4D_0$.
For each $i$ in $I_1$ (\resp $I_2$) and each $\gamma\in\Gamma$, let $Y_{\rho_i(\gamma)}^-$ be given by Proposition~\ref{prop:Lip-proj} (\resp Proposition~\ref{prop:Lip-hyp} for $\eps/2$ instead of~$\eps$).
Then there exists $D\geq D_0\geq 1$ such that
\begin{enumerate}[(a)]
  \item\label{item:a-bis} for any $i\in I_1\cup I'_2$ and any $\gamma\in\Gamma$, the restriction of $\rho_i(\gamma)$ to $B^{\eps/2}_{Y_{\rho_i(\gamma)}^-}$ is $D$-Lipschitz.
\end{enumerate}

Choose $\eps'>0$ small enough so that $2D^3\eps' \leq \eps$.
Up to replacing $\gamma_1$ by a large power, and then $\gamma_0$ by a large power, we may assume that
\begin{enumerate}[(a)] \setcounter{enumi}{1}
  \item\label{item:b-bis} $\rho_i(\gamma_1)$ is $(r,\eps')$-proximal for all $i\in I$, and $\rho_{i_0}(\gamma_0)$ is $(r,\eps'/L_1)$-proximal where $L_1 := \max(\mathrm{Lip}(\rho_{i_0}(\gamma_1)), \mathrm{Lip}(\rho_{i_0}(\gamma_1)^{-1}), 1)$.
\end{enumerate}
By Corollary~\ref{cor:finiteset} with $\cY_i^- = \{ Y_{\rho_i(\gamma_0)}^-\}$, we can find $\beta'\in\mathcal{F}$ such that
\begin{enumerate}[(a)] \setcounter{enumi}{2}
  \item\label{item:c-bis} $d_i(\rho_i(\beta')\cdot x_i^+, Y_{\rho_i(\gamma_0)}^-) \geq 3r$ for all $i\in I$,
\end{enumerate}
and by Corollary~\ref{cor:finiteset} with $\cY_i^+ = \{\rho_i(\gamma_0\beta')\cdot x_i^+\}$ for $i\in I$ and $\cY_{i_0}^- = \{\rho_{i_0}(\beta\gamma_1)^{-1}\cdot X_{i_0}^-\}$, we can find $\beta\in\mathcal{F}$ such that
\begin{enumerate}[(a)] \setcounter{enumi}{3}
  \item\label{item:d-bis} $d_i(\rho_i(\gamma_0\beta')\cdot x_i^+, \rho_i(\beta)^{-1}\cdot X_i^-) \geq 3r$ for all $i\in I$, and $d_{i_0}(\rho_{i_0}(\beta)\cdot x_{i_0}^+, \rho_{i_0}(\beta\gamma_1)^{-1}\cdot X_{i_0}^-) \geq 3r$. 
\end{enumerate}

We now consider the element
$$\gamma := \beta \gamma_0 \beta' \gamma_1 \in \Gamma.$$
Let us check that $\rho_i(\gamma)$ is proximal for all $i\in I\cup\{i_0\}$.

We claim that for any $i\in I$ the element $\rho_i(\gamma)$ sends $B^{\eps/2}_{X_i^-}$ into $b^{\eps/2}_{\rho_i(\beta\gamma_0\beta')\cdot x_i^+}$ in an $(\eps/2)$-Lipschitz way.
Indeed, this is a direct consequence of the following four observations:
\begin{enumerate}
  \item by \eqref{item:b-bis}, the element $\rho_i(\gamma_1)$ sends $B^{\eps/2}_{X_i^-}$ into $b^{\eps'}_{x_i^+}$ in an $\eps'$-Lipschitz way;
  \item $\rho_i(\beta')$ is $D$-Lipschitz, hence sends $b^{\eps'}_{x_i^+}$ into $b^{D\eps'}_{\rho_i(\beta')\cdot x_i^+}$;
  \item by \eqref{item:c-bis}, since $D\eps' + \eps/2 \leq 3r$, we have $b^{D\eps'}_{\rho_i(\beta')\cdot x_i^+} \subset B^{\eps/2}_{Y_{\rho_i(\gamma_0)}^-}$, and so by \eqref{item:a-bis} the element $\rho_i(\gamma_0)$ is $D$-Lipschitz on $b^{D\eps'}_{\rho_i(\beta')\cdot x_i^+}$ and sends it into $b^{D^2\eps'/2}_{\rho_i(\gamma_0\beta')\cdot x_i^+}$;
  \item $\rho_i(\beta)$ is $D$-Lipschitz, hence sends $b^{D^2\eps'/2}_{\rho_i(\gamma_0\beta')\cdot x_i^+}$ into $b^{D^3\eps'/2}_{\rho_i(\beta\gamma_0\beta')\cdot x_i^+} \subset b^{\eps/2}_{\rho_i(\beta\gamma_0\beta')\cdot x_i^+}$.
\end{enumerate}
By \eqref{item:d-bis} and the fact that $\rho_i(\beta)$ is $D_0$-Lipschitz, we have $d_i(\rho_i(\beta\gamma_0\beta')\cdot x_i^+, X_i^-) \geq 3r/D_0$.
By Lemmas \ref{l.critprox-proj} and~\ref{l.critprox-hyp} with $(r/(2D_0),\eps/2)$ instead of $(r,\eps)$, the element $\rho_i(\gamma)$ is $(r/D_0,\eps)$-proximal for all $i\in I$.

We claim that $\rho_{i_0}(\gamma)$ sends $B^{\eps/2}_{\rho_{i_0}(\beta\gamma_1)^{-1}\cdot X_{i_0}^-}$ into $b^{\eps/2}_{\rho_{i_0}(\beta)\cdot x_{i_0}^+}$ in an $(\eps/2)$-Lipschitz way.
Indeed, this is a direct consequence of the following four observations:
\begin{enumerate}
  \item $\rho_{i_0}(\beta\gamma_1)$ and $\rho_{i_0}(\beta\gamma_1)^{-1}$ are $(DL_1)$-Lipschitz, hence $\rho_{i_0}(\beta\gamma_1)$ sends $B^{\eps/2}_{\rho_{i_0}(\beta\gamma_1)^{-1}\cdot X_{i_0}^-}$ into $B^{\eps/(2DL_1)}_{X_{i_0}^-} \subset B^{\eps'/L_1}_{X_{i_0}^-}$;
   \item by \eqref{item:b-bis}, the element $\rho_{i_0}(\gamma_0)$ sends $B^{\eps'/L_1}_{X_{i_0}^-}$ into $b^{\eps'/L_1}_{x_{i_0}^+}$ in an $(\eps'/L_1)$-Lipschitz way;
  \item $\rho_{i_0}(\beta)$ is $D$-Lipschitz, hence sends $b^{\eps'/L_1}_{x_{i_0}^+}$ into $b^{D\eps'/L_1}_{\rho_{i_0}(\beta)\cdot x_{i_0}^+} \subset b^{\eps/2}_{\rho_{i_0}(\beta)\cdot x_{i_0}^+}$.
\end{enumerate}
By \eqref{item:d-bis} and Lemmas \ref{l.critprox-proj} and~\ref{l.critprox-hyp} with $(r/2,\eps/2)$ instead of $(r,\eps)$, the element $\rho_{i_0}(\gamma)$ is $(r,\eps)$-proximal.

This shows, by induction, that if $I_1\cup I'_2 \neq \emptyset$, then there exists $\gamma'\in\Gamma$ such that $\rho_i(\gamma')$ is proximal for all $i\in I_1\cup I'_2$.

$\bullet$ \textbf{Step 2: the general case.}
If $I_2 = I'_2$, then we are done by Step~1, so we now assume $I_2 \smallsetminus I'_2 \neq \emptyset$.
By Lemma~\ref{lem.linealcase}, there exists $C>0$ such that for any $i\in I_2\smallsetminus I'_2$, any element $\gamma \in \Gamma$ satisfying $d_{M_i}(o_i,\rho_i(\gamma)\cdot o_i) > C$ is $(r,\eps)$-proximal in $\partial_{\infty}M_i$.
By Lemma~\ref{lem:simult-basept-infty}, there exists an element $\gamma_0\in\Gamma$ such that $\rho_i(\gamma_0)$ is proximal for all $i\in I_2\smallsetminus I'_2$.
If $I_1\cup I'_2 = \emptyset$, then we are done, so we now assume $I_1\cup I'_2 \neq \emptyset$.
By Step~1, there exists $\gamma'\in\Gamma$ such that $\rho_i(\gamma')$ is proximal for all $i\in I_1\cup I'_2$.
Let $x_i^+$ (\resp $X_i^-$) be the attractor (\resp repellor) of $\rho_i(\gamma')$.

For $\cX_i^+ = \{ x_i^+\}$ and $\cX_i^- = \{ X_i^-\}$, with $N'=1$, let $r_0>0$ be the constant and $\mathcal{F}$ the finite subset of~$\Gamma$ given by Corollary~\ref{cor:finiteset} for $I_1\cup I'_2$.
Let $r_1 := \min_{i\in I_1\cup I'_2} d_i(x_i^+,X_i^-) > 0$.
As in the proof of Theorem~\ref{teo.generalresult}, there exists $D_0\geq 1$ such that $\rho_i(\beta)$ and $\rho_i(\beta)^{-1}$ are $D_0$-Lipschitz for all $i\in I_1\cup I_2$ and all $\beta\in\mathcal{F}$.
Fix $0 < \eps \leq r \leq \min(r_0,r_1)/4D_0$.

Choose $\eps'>0$ small enough so that $2D^3\eps' \leq \eps$.
Up to replacing $\gamma'$ by a large power, and then~$\gamma_0$ by a large power, we may assume that
\begin{enumerate}[(a)] \setcounter{enumi}{1}
  \item\label{item:b-ter} $\rho_i(\gamma')$ is $(r,\eps')$-proximal for all $i\in I_1\cup I'_2$ and $d_{M_i}(o_i,\rho_i(\gamma_0)) > C + d_{M_i}(o_i,\rho_i(\gamma')) + 2\,\max_{\beta''\in\mathcal{F}} d_{M_i}(o_i,\rho_i(\beta'')\cdot o_i)$ for all $i\in I_2\smallsetminus I'_2$.
\end{enumerate}
By Corollary~\ref{cor:finiteset}, we can find $\beta'\in\mathcal{F}$ such that
\begin{enumerate}[(a)] \setcounter{enumi}{2}
  \item\label{item:c-ter} $d_i(\rho_i(\beta')\cdot x_i^+, Y_{\rho_i(\gamma_0)}^-) \geq 3r$ for all $i\in I_1\cup I'_2$,
\end{enumerate}
and we can find $\beta\in\mathcal{F}$ such that
\begin{enumerate}[(a)] \setcounter{enumi}{3}
  \item\label{item:d-ter} $d_i(\rho_i(\gamma_0\beta')\cdot x_i^+, \rho_i(\beta)^{-1}\cdot X_i^-) \geq 3r$ for all $i\in I_1\cup I'_2$. 
\end{enumerate}

We now consider the element
$$\gamma := \beta \gamma_0 \beta' \gamma' \in \Gamma.$$
Exactly as in Step~1, for any $i\in I_1\cup I'_2$, the element $\rho_i(\gamma)$ sends $B^{\eps/2}_{X_i^-}$ into $b^{\eps/2}_{\rho_i(\beta\gamma_0\beta')\cdot x_i^+}$ in an $(\eps/2)$-Lipschitz way, hence is $(r/D_0,\eps)$-proximal.
On the other hand, for any $i\in I_2\smallsetminus I'_2$, by \eqref{item:b-ter} and the triangle inequality (see \eqref{eqn:triangle-ineq-dist-o}) we have $d_{M_i}(o_i,\rho_i(\gamma)) > C$, hence $\rho_i(\gamma)$ is proximal.
\end{proof}

\section{Simultaneous proximality in flag varieties of real reductive groups and in Gromov boundaries of hyperbolic metric spaces} \label{sec:proof-cor-main}

In this section we use Theorem~\ref{thm:main} to prove a variant of it (Theorem~\ref{thm:prox-in-G/P}) where products of projective spaces are replaced by flag varieties of real reductive linear Lie groups; we also prove Corollary~\ref{cor:main}.

\subsection{A variant of Theorem~\ref{thm:main} for flag varieties} \label{subsec:flag-varieties}

Let $G$ be a noncompact connected real reductive linear Lie group, and $P$ a minimal parabolic subgroup of~$G$.
An element $g\in G$ is said to be \emph{proximal in $G/P$} if it admits an attracting fixed point $x_g^+$ in $G/P$, \ie a fixed point $x_g^+$ with a neighborhood $\cV_g$ in $G/P$ such that $g^n\cdot x\to x_g^+$ as $n\to +\infty$ for all $x\in\cV_g$.
In that case, $x_g^+$ is unique, $g^{-1}$ is also proximal, and we can take $\cV_g = G/P \smallsetminus X_g^-$ where $X_g^-$ is the set of points of $G/P$ that are \emph{not} transverse to $x_g^- := x_{g^{-1}}^+$ (this set is a proper algebraic subvariety of $G/P$).

Endow $G/P$ with a Riemannian metric $d_{G/P}$.
For any $\eps>0$, any point $x\in G/P$, and any closed subset $X$ of $G/P$, we denote by $b_x^{\eps}$ the closed ball of radius $\eps$ centered at~$x$, and by $B_X^{\eps}$ the set of points of $G/P$ at distance at least $\eps$ from all points of~$X$.
Following Definition~\ref{def:r-eps-prox-proj}, for $r\geq\eps>0$ we say that a proximal element $g\in\GL(V)$ is \emph{$(r,\eps)$-proximal in $G/P$} if it satisfies the following three conditions:
\begin{enumerate}
  \item $d_{G/P}(x_g^+,X_g^-)\geq 2r$,
  \item $g \cdot B_{X_g^-}^{\eps} \subset b_{x_g^+}^{\eps}$,
  \item the restriction of $g$ to $B_{X_g^-}^{\eps}$ is $\varepsilon$-Lipschitz.
\end{enumerate}

Theorem~\ref{thm:main} has the following consequence.

\begin{teo} \label{thm:prox-in-G/P}
Let $G$ be a noncompact connected real reductive linear Lie group, $P$ a minimal parabolic subgroup of~$G$, and $\Gamma$ a Zariski-dense subsemigroup of~$G$.
Let $I$ be a finite set, and for each $i\in I$, let $M_i$ be a Gromov hyperbolic metric space with a choice of Bourdon metric on $\partial_{\infty}M_i$, and let $\rho_i : \Gamma\to\Isom(M_i)$ be a representation such that $\rho_i(\Gamma)$ acts on $\partial_{\infty}M_i$ without a unique global fixed point and contains an element which is proximal in $\partial_{\infty}M_i$.
Then there exists $r'_0>0$ such that for any $r'_0\geq r\geq\eps>0$, there is a finite subset $S$ of~$\Gamma$ with the following property: for any $\gamma \in \Gamma$, we can find $s\in S$ such that $\gamma s$ is $(r,\eps)$-proximal in $G/P$ and $\rho_i(\gamma s)$ is $(r,\eps)$-proximal in $\partial_{\infty}M_i$ for all $i\in I$.
\end{teo}

\subsection{Proof of Theorem~\ref{thm:prox-in-G/P}} \label{subsec:proof-prox-in-G/P}

Let $\aaa_G$ be a Cartan subspace of the Lie algebra $\g$ of~$G$, let $\Sigma = \Sigma(\g,\aaa_G)$ be the set of restricted roots of $\aaa_G$ in~$\g$, and let $\Delta \subset \Sigma$ be a system of simple restricted roots.
As observed by Tits, for any $\alpha\in\Delta$ there is an irreducible linear representation $(\tau_{\alpha},V_{\alpha})$ of~$G$ whose highest weight $\chi_{\tau_{\alpha}}$ is a multiple of the fundamental weight~$\omega_{\alpha}$ associated with~$\alpha$ and such that $\tau_{\alpha}(G)$ contains an element which is proximal in $\PP(V_{\alpha})$ (see \eg \cite[Lem.\,3.5]{GGKW}).
Since $\Gamma$ is Zariski-dense in~$G$, for every $\alpha\in\Delta$, the semigroup $\tau_{\alpha}(\Gamma)$ acts strongly irreducibly on~$V_{\alpha}$; moreover, $\tau_{\alpha}(\Gamma)$ contains an element which is proximal in $\PP(V_{\alpha})$ (see \cite{gm89,gr89,bl93,pra94}).
Finally, an element $g\in G$ is proximal in $G/P$ if and only if $\tau_{\alpha}(g)$ is proximal in $\PP(V_{\alpha})$ for all $\alpha\in\Delta$ (see \eg \cite[Prop.\,3.3.(c)]{GGKW}).

For every $\alpha\in\Delta$, fix a Euclidean structure on~$V_{\alpha}$ and let $d_{\PP(V_{\alpha})}$ be the corresponding distance function on $\PP(V_{\alpha})$ from Section~\ref{sec:proj}.
There are $G$-equivariant embeddings
$$\psi = (\psi_{\alpha})_{\alpha\in\Delta} : G/P \longhookrightarrow \prod_{\alpha\in\Delta} \PP(V_{\alpha}) \quad\quad\mathrm{and}\quad\quad \psi^* = (\psi^*_{\alpha})_{\alpha\in\Delta} : G/P \longhookrightarrow \prod_{\alpha\in\Delta} \PP(V_{\alpha}^*)$$
such that
\begin{enumerate}[(a)]
  \item\label{item:psi-1} if we endow $\prod_{\alpha\in\Delta} \PP(V_{\alpha})$ with the sup metric associated with the $d_{\PP(V_{\alpha})}$, then $\psi : G/P\to\psi(G/P)$ is $M$-bi-Lipschitz for some $M\geq 1$;
  \item\label{item:psi-2} two points $x,y\in G/P$ are transverse if and only if $\psi_{\alpha}(x)$ and $\psi^*_{\alpha}(y)$ are for all $\alpha\in\Delta$;
  \item\label{item:psi-3} if $g\in G$ is proximal in $G/P$, then for any $\alpha\in\Delta$ the attracting fixed point of $\tau_{\alpha}(g)$ in $\PP(V_{\alpha})$ is $\psi_{\alpha}(x_g^+)$, and its repelling projective hyperplane is $\psi^*_{\alpha}(x_g^-) = \psi_{\alpha}(X_g^-)$
\end{enumerate}
(see \eg \cite[\S\,3.4]{ben97}).

\begin{lema} \label{lem:r-eps-prox-G/P}
For any $\eps>0$, there exists $\eps' \in (0,\eps]$ with the following property: for any $g\in G$ and any $r\geq\eps$, if $\tau_{\alpha}(g)$ is $(Mr,\eps')$-proximal in $\PP(V_{\alpha})$ for all $\alpha\in\Delta$, then $g$ is $(r,\eps)$-proximal in $G/P$.
\end{lema}

\begin{proof}
For $y\in G/P$, we denote by $\cZ_y$ the set of points of $G/P$ that are \emph{not} transverse to~$y$ (a proper algebraic subvariety of $G/P$).

Fix $\eps>0$.
Note that the set of $(x,y)\in (G/P)^2$ with $d_{G/P}(z,\cZ_y)\geq\eps$ is compact.
Therefore, by \eqref{item:psi-2} and by continuity of $\psi_{\alpha}$ and~$\psi^*_{\alpha}$, there exists $\eps' \in (0,\eps/M^2]$ such that $d_{\PP(V_{\alpha})}(\psi_{\alpha}(x),\psi^*_{\alpha}(y)) \geq \eps'$ for all $\alpha\in\Delta$ and all $x,y\in G/P$ with $d_{G/P}(x,\cZ_y)\geq\eps$ (where we see $\psi^*_{\alpha}(y)$ as a projective hyperplane in $\PP(V_{\alpha})$).

Consider $g\in G$ and $r\geq\eps$ such that $\tau_{\alpha}(g)$ is $(Mr,\eps')$-proximal in $\PP(V_{\alpha})$ for all $\alpha\in\Delta$.
By \eqref{item:psi-3} we have $d_{\PP(V_{\alpha})}(\psi_{\alpha}(x_g^+),\psi_{\alpha}(X_g^-))\geq 2Mr$ for all $\alpha\in\Delta$, hence $d_{G/P}(x_g^+,X_g^-)\geq 2r$ by~\eqref{item:psi-1}.
For any $\alpha\in\Delta$, we have $\psi_{\alpha}(B_{X_g^-}^{\eps}) \subset B_{\psi^*_{\alpha}(x_g^-)}^{\eps'}$ by choice of~$\eps'$, hence $\psi_{\alpha}(g\cdot B_{X_g^-}^{\eps}) = g\cdot\psi_{\alpha}(B_{X_g^-}^{\eps}) \subset b_{\psi_{\alpha}(x_g^+)}^{\eps'} \subset b_{\psi_{\alpha}(x_g^+)}^{\eps/M}$ by $G$-equivariance of~$\psi_{\alpha}$ and~\eqref{item:psi-3}; in particular, $g\cdot B_{X_g^-}^{\eps} \subset b_{x_g^+}^{\eps}$ by~\eqref{item:psi-1}.
Moreover, for any $\alpha\in\Delta$ the restriction of $g$ to $\psi_{\alpha}(B_{X_g^-}^{\eps}) \subset B_{\psi^*_{\alpha}(x_g^-)}^{\eps'}$ is $\eps'$-Lipschitz, hence $(\eps/M^2)$-Lipschitz, and so the restriction of $g$ to $B_{X_g^-}^{\eps}$ is $\eps$-Lipschitz by~\eqref{item:psi-1}.
\end{proof}

\begin{proof}[Proof of Theorem~\ref{thm:prox-in-G/P}]
Let $r_0>0$ be given by Theorem~\ref{thm:main}, and let $M\geq 1$ be given by \eqref{item:psi-1} above.
We claim that we may take $r'_0 = r_0/M$.
Indeed, consider $r'_0\geq r\geq\eps>0$.
Let $\eps' \in (0,\eps]$ be given by Lemma~\ref{lem:r-eps-prox-G/P}.
By Theorem~\ref{thm:main}, there is a finite subset $S$ of~$\Gamma$ with the property that for any $\gamma \in \Gamma$, we can find $s\in S$ such that $\tau_{\alpha}(\gamma s)$ is $(Mr,\eps')$-proximal in $\PP(V_{\alpha})$ for all $\alpha\in\Delta$ and $\rho_i(\gamma s)$ is $(Mr,\eps')$-proximal in $\partial_{\infty}M_i$ for all~$i$.
Then $\gamma s$ is $(r,\eps)$-proximal in $G/P$ by Lemma~\ref{lem:r-eps-prox-G/P} and $\rho_i(\gamma s)$ is $(r,\eps)$-proximal in $\partial_{\infty}M_i$ for all~$i$ by Remark~\ref{rem:r-eps-r'-eps'-prox-hyp}.
\end{proof}

\subsection{A simultaneous control of lengths}

We now prove Corollary~\ref{cor:main}.
For this, we use the fact that for any Euclidean space~$V$ and any $g,g_1,g_2\in\GL(V)$,
\begin{equation} \label{eqn:mu-subadd}
\| \mu(g_1 g g_2) - \mu(g) \| \leq \| \mu(g_1) \| + \| \mu(g_2) \|
\end{equation}
(where $\mu : \GL(V)\to\RR^{\mathtt{d}}$ is the Cartan projection of Section~\ref{subsec:intro-classical-AMS}), and that for any Gromov hyperbolic metric space $(M,d_M)$ with basepoint $o\in M$ and any $g,g_1,g_2\in\Isom(M)$,
\begin{equation} \label{eqn:length-subadd}
| |g_1 g g_2|_M - |g|_M | \leq |g_1|_M + |g_2|_M
\end{equation}
(with the notation \eqref{eqn:length-stable-length}).
For \eqref{eqn:mu-subadd}, see \eg \cite[Lem.\,2.3]{kas08}.
For \eqref{eqn:length-subadd}, note that $|g_1 g g_2|_M = d_M(o,g_1 g g_2\cdot o) \leq d_M(o,g_1\cdot o) + d_M(g_1\cdot o,g_1g\cdot o) + d_M(g_1g\cdot o, g_1gg_2\cdot o) = |g_1|_M + |g|_M + |g_2|_M$, and similarly $|g|_M \leq |g_1^{-1}|_M + |g_1gg_2|_M + |g_2^{-1}|_M = |g_1|_M + |g_1gg_2|_M + |g_2|_M$.

\begin{proof}[Proof of Corollary~\ref{cor:main}]
Write $M$ as a direct product of finitely many Gromov hyperbolic metric spaces~$M_i$.
By assumption, for each $i$ the semigroup $\rho_i(\Gamma)$ is either elliptic, lineal, or of general type (see Section~\ref{subsubsec:isom-hyp}).
Let $I_2$ denote the set of $i$ for which $\rho_i(\Gamma)$ is lineal or of general type.

By \eqref{eqn:mu-subadd} and \eqref{eqn:length-subadd}, it is sufficient to prove the existence of a finite set $S\subset\Gamma$ and a constant $C>0$ such that for any $\gamma \in \Gamma$ we can find $s\in S$ with $\|\lambda(\rho(\gamma s)) - \mu(\rho(\gamma s))\| \leq C$ and $| |\gamma s|_{M_i,\infty} - |\gamma s|_{M_i}| \leq C$ for all~$i$.
In fact we can restrict to $i\in I_2$ since for $i\notin I_2$ the function $|\cdot|_{M_i}$ is uniformly bounded on~$\Gamma$ and $|\cdot|_{M_i,\infty}$ is zero.

By assumption, the Zariski closure $G$ of $\rho(\Gamma)$ in $\GL(V)$ is reductive.
It admits a Cartan decomposition $G = K_G \exp(\aaa^+_G) K_G$ where $K_G$ is a maximal compact subgroup of~$G$ and $\aaa^+_G$ is a closed Weyl chamber in a Cartan subspace $\aaa_G$ of the Lie algebra $\g$ of~$G$.
Any element $g\in G$ can be written as $g = k\exp(a)k'$ for some $k,k'\in K_G$ and a unique $a\in\aaa^+_G$; setting $\mu_G(g) := a$ defines a \emph{Cartan projection} $\mu_G : G\to\aaa^+_G$.
We also have a \emph{Jordan projection} (or \emph{Lyapunov projection}) $\lambda_G : G\to\aaa^+_G$ such that $\lambda_G(g) = \lim_n \mu_G(g^n)/n$ for all $g\in G$.

Up to replacing $G$ by a conjugate in $\GL(V)$ (which only modifies $\mu$ and $|\cdot|_M$ by a bounded additive amount, see \eqref{eqn:mu-subadd} and \eqref{eqn:length-subadd}), we may assume that the Cartan decomposition $G = K_G\exp(\aaa_G^+)K_G$ is compatible with the Cartan decomposition $\GL(V) = K\exp(\aaa)K$ of Section~\ref{subsec:unif-contract-proj}, in the sense that $K_G \subset K$ and $\aaa_G \subset \aaa$.
It is then sufficient to prove the existence of a finite set $S\subset\Gamma$ and a constant $C>0$ such that for any $\gamma \in \Gamma$ we can find $s\in S$ with $\|\lambda_G(\rho(\gamma s)) - \mu_G(\rho(\gamma s))\| \leq C$ and $| |\gamma s|_{M_i,\infty} - |\gamma s|_{M_i}| \leq C$ for all $i\in I_2$, where $\Vert\cdot\Vert$ is any fixed norm on~$\aaa_G$.

We can write $\aaa_G = \aaa_G^z \oplus \aaa_G^s$, where $\aaa_G^z$ is the intersection of $\aaa_G$ with the Lie algebra of the center of~$G$, and $\aaa_G^s$ is a Cartan subspace of the Lie algebra of the derived subgroup (semisimple part) of~$G$.
Let $\Sigma = \Sigma(\g,\aaa_G)$ be the set of restricted roots of $\aaa_G$ in~$\g$, and let $\Delta \subset \Sigma$ be a system of simple restricted roots.
As in Section~\ref{subsec:proof-prox-in-G/P}, for any $\alpha\in\Delta$ there is an irreducible linear representation $(\tau_{\alpha},V_{\alpha})$ of~$G$ whose highest weight $\chi_{\tau_{\alpha}}$ is of the form $n_{\alpha}\omega_{\alpha}$ for some $n_{\alpha}\in\NN^*$ and such that $\tau_{\alpha}(G)$ contains an element which is proximal in $\PP(V_{\alpha})$; Theorem~\ref{thm:main} gives the existence of $r\geq\eps>0$ and of a finite subset $S$ of~$\Gamma$ with the property that for any $\gamma \in \Gamma$, we can find $s\in S$ such that $\tau_{\alpha}\circ\rho(\gamma s)$ is $(r,\eps)$-proximal in $\PP(V_{\alpha})$ for all $\alpha\in\Delta$ and $\gamma s$ is $(r,\eps)$-proximal in $\partial_{\infty}M_i$ for all $i\in I_2$.
By Corollary~\ref{cor:r-eps-prox-lambda-mu}, we have $|\langle n_{\alpha}\omega_{\alpha},\mu_G(\rho(\gamma s))-\lambda_G(\rho(\gamma s))\rangle| \leq |\log 2r^2|$ for all $\alpha\in\Delta$.
Since the fundamental weights $\omega_{\alpha}$, for $\alpha\in\Delta$, restrict to a basis of $(\aaa_G^s)^{\ast}$, and since for any $g\in G$ the elements $\lambda_G(g)$ and $\mu_G(g)$ of~$\aaa_G$ have the same projection to $\aaa_G^z$, we deduce $\Vert\mu_G(\rho(\gamma s))-\lambda_G(\rho(\gamma s))\Vert \leq C$ for some $C>0$ independent of $\gamma$ and~$s$.
On the other hand, by Corollary~\ref{cor:r-eps-prox-length}, for any $i\in I_2$ there exists $C_i>0$ such that for any $g\in\Isom(M_i)$ which is $(r,\varepsilon)$-proximal in $\partial_{\infty}M_i$ we have $||g|_{M_i,\infty} - |g|_{M_i}| \leq C_i$.
\end{proof}





\begin{thebibliography}{GGKW}

\bibitem[AMS]{ams95}
\textsc{H. Abels, G. A. Margulis, G. A. Soifer}, \textit{Semigroups containing proximal linear maps}, Israel J. Math.~91 (1995), p.~1--30.

\bibitem[Be]{ben97}
\textsc{Y. Benoist}, \textit{Propri\'et\'es asymptotiques des groupes lin\'eaires}, Geom. Funct. Anal.~7 (1997), p.~1--47.

\bibitem[Be$_2$]{Benoist-notes}
\textsc{Y. Benoist}, \textit{Sous-groupes discrets des groupes de Lie}, notes from the European Summer School in Group Theory, Luminy, 1997, available at \texttt{https://www.math.u-psud.fr/~benoist/prepubli/0097luminy.pdf}

\bibitem[BL]{bl93}
\textsc{Y. Benoist, F. Labourie}, \textit{Sur les diff\'eomorphismes d'Anosov affines \`a feuilletages stable et instable diff\'erentiables}, Invent. Math.~111 (1993), p.~285--308.

\bibitem[BQ]{BenoistQuint-livre}
\textsc{Y. Benoist, J. F. Quint}, \textit{Random walks on reductive groups}, Ergebnisse der Mathematik und ihrer Grenzgebiete, vol.~62, Springer, Cham, 2016.

\bibitem[BPS]{BPS}
\textsc{J. Bochi, R. Potrie, A. Sambarino}, \textit{Anosov representations and dominated splittings}, J. Eur. Math. Soc.~21 (2019), p.~3343--3414.

\bibitem[BT]{BorelTits}
\textsc{A. Borel, J. Tits}, \textit{Groupes r\'eductifs}, Publ. Math. Inst. Hautes \'Etudes Sci.~27 (1965), p.~55--151.

\bibitem[BHW]{BHW}
\textsc{J. Bowden, S. Hensel, R. Webb}, \textit{Quasi-morphisms on surface diffeomorphism groups}, J. Amer. Math. Soc. 35 (2022), p.~211--231.

\bibitem[BG]{BrGe}
\textsc{E. Breuillard, T. Gelander}, \textit{On dense free subgroups of Lie groups}, J. Algebra~261 (2003), p.~448--467.

\bibitem[BS]{BS}
\textsc{E. Breuillard, C. Sert}, \textit{The joint spectrum}, J. Lond. Math. Soc.~103 (2021), p.~943--990.

\bibitem[BH]{BH}
\textsc{M. R. Bridson, A. Haefliger}, \textit{Metric spaces of non-positive curvature}, Grundlehren der mathematischen Wissenschaften, vol.~319, Springer-Verlag, Berlin, 1999.

\bibitem[CT]{CaTs}
\textsc{R. Canary, K. Tsouvalas}, \textit{Topological restrictions on Anosov representations}, J. Topol.~13 (2020), p.~1497--1520.

\bibitem[Coo]{Coo}
\textsc{M. Coornaert}, \textit{Mesures de Patterson-Sullivan sur le bord d'un espace hyperbolique au sens de Gromov}, Pacific J. Math.~159 (1993), p.~241--270.

\bibitem[CCMT]{CCMT} \textsc{P.-E. Caprace, Y. de Cornulier, N. Monod, R. Tessera} \textit{Amenable hyperbolic groups}, J. Eur. Math. Soc.~17 (2015), p.~2903--2947.

\bibitem[DGK]{DGKp}
\textsc{J. Danciger, F. Gu\'eritaud, F. Kassel}, \textit{Proper affine actions for right-angled Coxeter groups}, Duke Math. J.~169 (2020), p.~2231--2280.

\bibitem[DSU]{dsu17}
\textsc{T. Das, D. Simmons, M. Urba\'nski}, \textit{Geometry and dynamics in Gromov hyperbolic metric spaces, with an emphasis on non-proper settings}, Mathematical Surveys and Monographs, vol.~218, American Mathematical Society, Providence, RI, 2017.

\bibitem[DGLM]{DGLM}
\textsc{M. Delzant, O. Guichard, F. Labourie, S. Mozes}, \textit{Displacing representations and orbit maps}, p.~494--514, in \textit{Geometry, rigidity, and group actions}, University of Chicago Press, 2011.

\bibitem[DP]{DP}
\textsc{T. Delzant, P. Py}, \textit{K\"{a}hler groups, real hyperbolic spaces and the Cremona group}, Compos. Math.~148 (2012), p.~153--184.

\bibitem[FSU]{FiSU}
\textsc{L. Fishman, D. Simmons, M. Urba\'{n}ski}, \textit{Diophantine approximation and the geometry of limit sets in Gromov hyperbolic metric spaces}, Memoirs of the American Mathematical Society, vol.~254, no. 1215, American Mathematical Society, Providence, RI, 2018.

\bibitem[GM]{gm89}
\textsc{I. Gol'dsheid, G. A. Margulis}, \textit{Lyapunov indices of a product of random matrices}, Russ. Math. Surveys~44 (1989), p.~11--71.

\bibitem[G]{Gou}
\textsc{S. Gou\"{e}zel}, \textit{Exponential bounds for random walks on hyperbolic spaces without moment conditions}, Tunisian J. Math.~4 (2022), p.~635--671.  

\bibitem[GGKW]{GGKW}
\textsc{F. Gu\'eritaud, O. Guichard, F. Kassel, A. Wienhard}, \textit{Anosov representations and proper actions}, Geom. Topol.~21 (2017), p.~485--584.

\bibitem[GR]{gr89}
\textsc{Y. Guivarc'h, G. Raugi}, \textit{Propri\'et\'es de contraction d'un semigroupe de matrices inversibles}, Israel J. Math.~65 (1989), p.~165--196.

\bibitem[HH]{hh20}
\textsc{B. B. Healy, G. C. Hruska}, \textit{Cusped spaces and quasi-isometries of relatively hyperbolic groups}, arXiv:2010.09876.

\bibitem[Ka]{kas08}
\textsc{F. Kassel}, \textit{Proper actions on corank-one reductive homogeneous spaces}, J. Lie Theory~18 (2008), p.~961--978.

\bibitem[KP]{kp22}
\textsc{F. Kassel, R. Potrie}, \textit{Eigenvalue gaps for hyperbolic groups and semigroups}, J. Mod. Dyn.~18 (2022), p.~161--208.

\bibitem[KSS]{kss}
\textsc{F. Kassel, \c{C}. Sert, I. Smilga}, in preparation.

\bibitem[KS]{ks}
\textsc{F. Kassel, I. Smilga}, \textit{Affine properness criteria for Anosov representations and generalizations}, in preparation.

\bibitem[Kn]{Knapp}
\textsc{A. Knapp}, \textit{Lie groups beyond an introduction}, second edition, Progress in Mathematics, vol.~140, Birkh\"auser, Boston, MA, 2002.

\bibitem[MT]{MT}
\textsc{J. Maher, G. Tiozzo}, \textit{Random walks on weakly hyperbolic groups}, J. Reine Angew. Math.~742 (2018), p.~187--239.

\bibitem[N]{Neumann}
\textsc{B. H. Neumann}, \textit{Groups covered by permutable subsets}, J. Lond. Math. Soc.~29 (1954), p.~236--248.

\bibitem[O]{OR}
\textsc{E. Oreg\'on-Reyes}, \textit{Properties of sets of isometries of Gromov hyperbolic spaces}, Groups Geom. Dyn. 12 (2018), p. 889--910.

\bibitem[P]{pra94}
\textsc{G. Prasad}, \textit{$\RR$-regular elements in Zariski-dense subgroups}, Q. J. Math.~45 (1994), p.~541--545.

\bibitem[T1]{tso24}
\textsc{K. Tsouvalas}, \textit{Cartan projections of fiber products and non-quasi-isometric embeddings.}
J. Lond. Math. Soc.~110 (2024), no. 5, Paper No. e70004.

\bibitem[T2]{tso26}
\textsc{K. Tsouvalas}, \textit{The H\"older exponent of Anosov limit maps}, J. Mod. Dyn.~22 (2026), p.~205--237.

\bibitem[Z1]{zhu21}
\textsc{F. Zhu}, \textit{Relatively dominated representations}, Ann. Inst. Fourier~71 (2021), p.~2169--2235.

\bibitem[Z2]{zhu23}
\textsc{F. Zhu}, \textit{Relatively dominated representations from eigenvalue gaps and limit maps}, Geom. Dedicata~217:39 (2023).

\bibitem[ZZ]{ZZ}
\textsc{F. Zhu, A. Zimmer} \textit{Relatively Anosov representations via flows I: theory}, to appear in Groups Geom. Dyn.

\end{thebibliography}
\end{document}